\documentclass[11pt]{amsart}
\pdfoutput=1
\usepackage{latexsym}
\usepackage{tabularx}
\usepackage{adjustbox}

\newcolumntype{P}[1]{>{\centering\arraybackslash}p{#1}}
\newcolumntype{M}[1]{>{\centering\arraybackslash}m{#1}}
\newcommand{\done}{\item[\checkmark]}
\newcommand{\crossed}{\item[$\times$]}
\usepackage{ifpdf} 
\usepackage{enumerate}
\usepackage{color}
\usepackage[utf8]{inputenc}
\usepackage{amssymb}
\usepackage{latexsym}
\usepackage{epsfig}
\usepackage{graphics}
\usepackage{graphicx}
\usepackage{subcaption}
\usepackage[update,prepend]{epstopdf}
\usepackage{cancel}
\usepackage[hyperindex,pdftex]{hyperref}
\usepackage{amsmath}
\usepackage{amsthm}
\usepackage{amsfonts}
\usepackage{amssymb}
\newtheorem{theorem}{Theorem}[section]
\newtheorem{lemma}[theorem]{Lemma}

\newtheorem{remark}{Remark}[section]

\newtheorem{corollary}[theorem]{Corollary}
\newtheorem{definition}[theorem]{Definition}

\newcommand{\set}[1]{{\dis\left\{#1\right\}}}

\newcommand{\norm}[2]{||{#1}||_{#2}}

\def\Om{{\Omega}}
\def\om{{\omega}}

\def\Re{\operatorname{{Re}}}
\def\Im{\operatorname{{Im}}}

\def\be{\begin{equation}}
\def\ee{\end{equation}}
\def\bq{\begin{eqnarray}}
\def\eq{\end{eqnarray}}
\def\bqs{\begin{eqnarray*}}
	\def\eqs{\end{eqnarray*}}
\def\wa{\begin{aligned}}
	\def\ali{\end{aligned}}
\def\ben{\begin{enumerate}}
	\def\enu{\end{enumerate}}
\def\bmu{\begin{multicols}}
	\def\emu{\end{multicols}}
\def\dis{\displaystyle}

\begin{document}

\begin{abstract}
In this paper, we study the defocusing nonlinear Schr\"{o}dinger equation with a locally distributed damping on a smooth bounded domain as well as on the whole space and on an exterior domain.  We first construct approximate solutions using the theory of monotone operators.  We show that approximate solutions decay exponentially fast in the $L^2$-sense by using the multiplier technique and a unique continuation property.  Then, we prove the global existence as well as the $L^2$-decay of solutions for the original model by passing to the limit and using a weak lower semicontinuity argument, respectively.  The distinctive feature of the paper is the monotonicity approach, which makes the analysis independent from the commonly used Strichartz estimates and allows us to work without artificial smoothing terms inserted into the main equation.
We in addition implement a precise and efficient algorithm for studying the exponential decay established in the first part of the paper numerically.  Our simulations illustrate the efficacy of the proposed control design.
\end{abstract}
	\author[Cavalcanti]{M.M. Cavalcanti }
	\address{ Department of Mathematics, State University of Maring\'a,
		87020-900, Maring\'a, PR, Brazil.}
	\email{mmcavalcanti@uem.br}
	\thanks{M.M. Cavalcanti's research was partially supported by the CNPq
		Grant 300631/2003-0}
	
	\author[Corr\^{e}a]{W.J. Corr\^{e}a}
	\address{ Federal Technological University of Paran\'{a}, Campus Campo Mour\~{a}o, 87301-006, Campo Mour\~{a}o, PR, Brazil.}
	\email{wcorrea@utfpr.edu.br}
        \thanks{W. J. Corrêa’s research was partially supported by the CNPq Grant
          438807/2018-9.}

	\author[\"{O}zsar\i]{T. \"{O}zsar\i}
	\address{Department of Mathematics,
	  Izmir Institute of Technology,
	  Izmir, 35430 Turkey}
	\email{turkerozsari@iyte.edu.tr}
	\thanks{T. Özsarı's research was supported by TÜBİTAK 1001 Grant \#117F449}

        \author[Sep\'ulveda]{M. Sep\'ulveda}
	\address{Centro de Investigación en Ingeniería Matemática \& Departamento de Ingenier\'ia matem\'atica, Universidad de Concepción, Chile
	}
        \email{mauricio@ing-mat.udec.cl}

	\author[V\'ejar]{R. V\'ejar-Asem }
	\address{Centro de Investigación en Ingeniería Matemática \& Departamento de Ingenier\'ia matem\'atica,, Universidad de Concepción, Chile
	}
        \email{rodrigovejar@ing-mat.udec.cl}
        \thanks{R. Véjar Asem acknowledges support by CONICYT-PCHA/Doctorado Nacional/2015-21150799.}
	
	\title[Stabilization of dNLS]{Exponential stability for the nonlinear Schr\"{o}dinger equation with locally distributed damping}
	
	\maketitle

	
	\section{Introduction}
	
	\setcounter{equation}{0}
	
	This paper is concerned with the stabilization of defocusing nonlinear Schr\"{o}dinger equations (dNLS)
	\begin{eqnarray}
	\label{problema}
	\begin{cases}
	i\,\partial_t y+\Delta
	y-|y|^p\,y+i\,a(x)\,y=0\quad \hbox{ in } \Omega\,\times (0,T),\\
	y(0)=y_0 \quad\hbox{ in }\Omega, \end{cases}
	\end{eqnarray} where $\Omega$ is a general domain, and $a$ is a nonnegative function that may vanish on some parts of the domain.  We first study  (dNLS) on a bounded domain $ \Omega $ in $\mathbb{R}^N$ with boundary $\Gamma $ of class $C^2$.  In this case we assume $y=0 \text{ on }\Gamma.$  Then, we extend the theory to unbounded domains in the particular cases $ \Omega=\mathbb{R}^N $ and $ \Omega $ being an exterior domain.

The nonlinear Schrödinger equation (NLS), central to classical field theory, gained fame when its one dimensional version was shown to be integrable in \cite{zak71}.  Contrary to its linear type, it does not describe the time evolution of a quantum state \cite{guo}.  It is rather used in other areas such as the transmission of light in nonlinear optical fibers and planar wavequides, small-amplitude gravity waves on the surface of deep inviscid water, Langmuir waves in hot plasmas, slowly varying packets of quasi-monochromatic waves in weakly nonlinear dispersive media, Bose-Einstein condensates, Davydov's alpha-helix solitons, and plane-diffracted wave beams in the focusing
regions of the ionosphere (see for instance \cite{Sulem99}, \cite{mal}, \cite{pitaevskii}, \cite{bala85}, and \cite{Gurevich}).

The NLS model without a damping term can describe an evolution without any mass and energy loss such as a laser beam propagated in the Kerr medium with no power losses.   However, it is always true that some absorption by the medium is indispensable even in the visible spectrum \cite{Fibich15}.  The effect of the absorption can be modelled by adding a linear (e.g., $iay$, $a>0$) or nonlinear (e.g., $ia|y|^qy$, $a>0$, $q>0$) damping term into the model, depending on the physical situation. A localized damping, where the damping coefficient $a=a(x)$ depends on the spatial coordinate, can be used to obtain better physical information by distinguishing the spatial region where the absorption takes place or is detected, due to for example some impurity in the medium, from the rest of the domain.
\subsection{Assumptions}  Throughout the paper (without any restatement) we will assume the following:
The power index $ p$ can be taken as any positive number. The nonnegative real valued function $a(\cdot)\in W^{1,\infty}(\Om)$ represents a localized	dissipative effect.

	\noindent If $\Omega$ is a bounded domain we will assume that
$a$ satisfies the geometric condition $a(x)\geq\,a_0>0$ (for some fixed $a_0\in \mathbb{R}_+$) for a.e. $x$ on a subregion $\omega\subset \Omega$  that contains $\overline{\Gamma(x^0)}$, where
	\begin{equation}\label{gamma0}
	\Gamma(x^0)=\{x\in\,\Gamma:\, m(x)\,\cdot\,\nu(x)>\,0\}.
	\end{equation} Here, $m(x):=x-x^0$ ($x_0\in \mathbb{R}^N$ is some fixed point), and $\nu(x)$
	represents the unit outward normal vector at the point $x\in \Gamma$.
	
	On the other hand, if $\Omega$ is the whole space, we assume $a(x) \geq a_0 >0$ in $\mathbb{R}^N \backslash B_{R^\prime}$, where $B_{R^\prime}$ represents a ball of radius $R^\prime>0$. We assume the same if $ \Omega $ is an exterior domain: { $ \Omega:=\mathbb{R}^N \setminus \mathcal{O}, $  where  $\mathcal{O} \subset\subset B_{R'}$ being $\mathcal{O}$ a compact star-shaped obstacle, namely, the following condition is verified:  $m(x)\cdot \nu(x) \leq 0$ on $ \Gamma_0,   $ where $\Gamma_0$  is the boundary of the obstacle $ \mathcal{O} $ which is }     smooth and associated with Dirichlet boundary condition  as in Lasiecka et al. \cite{las03}. {In this case, the observer $ x_0 $ must be taken in the interior of the obstacle $\mathcal{O}$. Regarding to the localized dissipative effect,  we consider $a(x) \geq  a_0 > 0$ in $\Omega\backslash B_{R'}$.}
	
	Moreover, {in all cases}, we assume that the damping coefficient $ a(\cdot) $ satisfies:
		\begin{equation}
		\label{anabla}|\nabla\,a(x)|^2 \lesssim \, a(x),\,\forall\,x\in\,\Omega.
		\end{equation}  The above assumption on the function $a(\cdot)$ was used for the wave equation with Kelvin-Voight damping; see for instance Liu \cite[Remark 3.1]{Liu} and Burq and Christianson \cite{Burq15}.
\begin{remark}
   The assumption $p>0$ is in parallel with the general theory of defocusing nonlinear Schrödinger equations when the initial datum is considered at the $H^1$-level.  On the contrary, it is well known that solutions of the focusing nonlinear Schrödinger equation (fNLS) may blow-up if $p\ge 4/N$ even in the presence of a weak damping acting on the whole domain for arbitrary initial data.  The main result of this paper can be extended to the case of the focusing problem via a Gagliardo-Nirenberg argument for the allowable range $p<4/N$.  The critical case $p= 4/N$ can also be treated with a smallness condition on the initial datum.  These are rather classical arguments and will be omitted here.
\end{remark}
%
%
\subsection{A few words on the previous work}
The stabilization problem for the linear and nonlinear Schrödinger equations (NLS) received significant attention in the last three decades. Tsutsumi \cite{Tsutsumi3} studied the stabilization of the weakly damped NLS posed on a bounded domain at the energy and higher levels. His results were extended to the weakly damped NLS posed on a bounded domain subject to inhomogeneous Dirichlet/Neumann boundary conditions in a series of papers by Özsarı  et al. \cite{Ozsar}, Özsarı \cite{Ozsari2}, \cite{Ozsari3}, and to the weakly damped NLS posed on the half-line subject to nonlinear boundary sources by Kalantarov \& Özsarı \cite{Kalantarov}.  In addition, Lasiecka \& Triggiani \cite{Lasiecka5} proved the exponential stability at the $L^2-$level for the linear Schrödinger equation with a nonlinear boundary dissipation.

\smallskip

In all of the work mentioned above, damping was assumed to be effective on the whole domain.  However, there has also been some progress regarding the stabilization with only a localized internal damping.  The stabilization problem in $L^2-$topology for the defocusing Schr\"odinger equation with a localized damping of the form $ia(x) y$ on the whole Euclidean space in dimensions one and two were treated by Cavalcanti et al. \cite{Cavalcanti}, \cite{Cavalcanti2},  \cite{Cavalcanti0}, and Natali \cite{Natali}, \cite{Natali1}.  Cavalcanti et al \cite{Cavalcanti4} considered an analogous structure of damping for the defocusing Schr\"odinger posed in a two dimensional compact Riemannian manifold without boundary.
Dehman et al. \cite{Dehman} studied the stabilization of the energy solutions for the defocusing cubic Schr\"odinger equations with a locally supported damping on a two dimensional boundaryless compact Riemannian manifold as well. For this purpose, the authors considered a damping term given by $i a(x) (I-\Delta)^{-1}a(x)\partial_t y$. Similar results on three dimensional compact manifolds were obtained by Laurent \cite{Laurent}.  Bortot et al. \cite{Bortot}  established uniform decay rate estimates for the Schr\"odinger equation posed on a compact Riemannian manifold subject to a locally distributed nonlinear damping. Bortot \& Cavalcanti \cite{Bortot1} extended these results to connected, complete and noncompact Riemannian manifolds.  Rosier \& Zhang \cite{Rosier1} obtained the local stabilization of the semilinear Schrödinger equation posed on $n$-dimensional rectangles. Burq \& Zworski \cite{burq2017rough} studied the exponential decay of the linear problem on 2 - Tori at the $ L^2 $ - level. In addition, we would like to cite Aloui et al. \cite{Aloui2}, who obtained the uniform stabilization of the strongly dissipative linear Schr\"odinger equation, and the recent work of Bortot \& Corr\^ea \cite{Bortot2017} for the treatment of the corresponding nonlinear model. It is worth mentioning that in  \cite{Aloui2} and \cite{Bortot2017} the authors considered a strong damping given by the structure $i a(x) (-\Delta)^{1/2} a(x) y$ which provides a local smoothing effect that was crucial in their proof.

\smallskip
\subsection{Motivation}
The main goal of the present paper is to achieve stabilization with the (natural) weaker dissipative effect $i a(x) y$ instead of relying on a strong dissipation such as $i a(x) (-\Delta)^{1/2} a(x) y$.  It will turn out that the assumption \eqref{anabla} enables us to avoid using such strong dissipation. We want to achieve stabilization in all dimensions $N\ge 1$ and for all power indices $p>0$.  For this purpose, we first construct approximate solutions to problem (\ref{probF2u}) by using the theory of monotone operators.  We show that these approximate solutions decay exponentially fast in the $L^2$-sense by using the multiplier technique and a unique continuation property.  Then, we prove the global existence as well as the $L^2$-decay of solutions for the original model by passing to the limit and using a weak lower semicontinuity argument, respectively. Here it should be noted that our nonlinear structure $ f(| y|^ 2)y $ ($f(s)=s^{p/2}$) is much more general than those treated to date in the context of stabilization with a locally supported damping. The current paper complements the work of Aloui et al. \cite{Aloui2} on unbounded domains, because we prove the global exponential decay for dNLS, while  \cite{Aloui2} obtained only a local exponential decay in the linear setting. In addition, we implement a precise and efficient algorithm for studying the exponential decay established in the first part of the paper numerically.  Our simulations illustrate the efficacy of the proposed control design.

\subsection{Main result}
 We adapt to the following notion of weak solutions for problem (\ref{problema}).
	 	 \begin{definition}\label{def1}
	 	 	Let $y_0\in L^2(\Omega)$ and set $\mathcal{X}=H_0^1(\Omega)\cap L^{p+2}(\Omega)$. Then,  $y\in\,L^{\infty}(0,T;\mathcal{X})\,\cap\,C([0,T];L^2(\Omega))$
	 	 	is said to be a \textit{weak solution} of problem (\ref{problema}) if $y$ satisfies $y(0,\cdot)=y_0(\cdot)$ in $L^2(\Omega)$, and
 	\begin{eqnarray}
	 	 	\label{weak1}&&\int_0^T \left[-( y(t),\partial_t\,\varphi(t))_{L^2(\Omega)}+i\,(\nabla\,y(t),\nabla\varphi(t))_{L^2(\Omega)}\right]\,dt\\
	 	 	\nonumber&&+i\int_0^T\left[\langle|\,y(t)\,|^{p}\,y(t),\varphi(t)\rangle_{L^{(p+2)'}(\Om);\, L^{p+2}(\Om)}-i(a(x)\,y(t),\varphi(t))_{L^2(\Omega)}\right]
	 	 	\,dt=0
	 	 	\end{eqnarray} for all $\varphi\in\,C_0^{\infty}(0,T;\mathcal{X})$.
	 	 \end{definition}
	 	 The mass functional for the defocusing NLS is given by $E_0(y(t)):=\frac12\,||y(t)||_{L^2\left(\Omega\right)}^2.$		 	
	 	 \begin{theorem}[Existence and stabilization]\label{theorem 2.1}
	 	 	Let $y_0\in \mathcal{X}=H_0^1(\Omega)\cap L^{p+2}(\Omega)$. Then, (\ref{problema}) admits a weak solution $y$ in the sense of Definition \ref{def1}.  Moreover, there are $C,\gamma>0$ (depending on $\norm{y_0}{H_0^1(\Omega)}$) such that the following exponential decay rate estimate
	 	 	\begin{eqnarray*}
	 	 		E_0(y(t)) \leq C e^{-\gamma t}E_0(y_0),t\geq T_0,
	 	 	\end{eqnarray*}  holds true for this weak solution provided $T_0>0$ is sufficiently large.
	 	 \end{theorem}	

  The proof of the exponential decay estimate as in Theorem \ref{theorem 2.1} is generally reduced to showing that given $R>0$, an inequality of the form
  \begin{equation}
	 	 		\int_{0}^{T}\int_{\Om\backslash \om}|y|^2dxdt\leq
	 	 		c\int_0^T\int_{\Om}a(x)|y|^2\,dx\,dt
	 	 		\label{estUCP3}\end{equation} must be satisfied for all $y$ solving \eqref{problema} with data satisfying $\|y_0\|_{\mathcal{X}}\le R$.  It is standard to prove these kinds of inequalities by contradiction, since then one can obtain a sequence of initial data satisfying $\|y_{0}^k\|_{\mathcal{X}}\le R$, whose corresponding solutions $y^k$ violate \eqref{estUCP3} with say $c=k$. The a priori bound $\|y_{0}^k\|_{\mathcal{X}}\le R$ is used to pass to a subsequence of $y^k$ which is expected to converge (in an appropriate sense) to a solution of the fully nonlinear model, say $u$, which in particular vanishes on $\omega$ (or on $\mathbb{R}^N \backslash B_{R^\prime}$ if $\Omega$ is unbounded).  Then a unique continuation argument must be triggered to conclude that $u$ is zero, which indicates a contradiction based on a further standard normalization argument.  Unfortunately, there is no established wellposedness theory for NLS when it is considered on a general domain with arbitrary data and power index, especially in dimensions three and higher. Absence of uniquesness and smoothing results for general domains makes it quite difficult to handle the nonlinear terms in passage to the limits and obtain a unique continuation property.  This motivates us to follow a novel strategy for stabilizing locally damped pdes based on first working with approximate models whose nonlinear parts are only Lipschitz.   The approximate models possess the desired uniqueness and strong regularity properties.  We focus on exponentially stabilizing solutions of these approximate models.  This is considerably easier than working with the fully nonlinear model because we can easily obtain a unique continuation property for the approximate models. The biggest advantage is that we do not need to handle highly nonlinear terms and therefore do not need to use smoothing properties generally implied by Strichartz type estimates, which are not widely available or true on general domains.  Once the exponential stability for approximate models is established, the existence of a weak solution as well as its exponential stability for the original model \eqref{problema} is achieved in a single shot.

\subsection{Orientation }	
The proof of  Theorem  \ref{theorem 2.1} requires a combination of several steps:
\begin{itemize}
  \item[Step 1:] We shall first work on a bounded domain and construct \emph{approximate solutions}. This is achieved by using the $m$-accretivity of the nonlinear source $By=|y|^py$ on a suitably chosen domain.  This allows us to replace $By$ with its Yosida approximations $B_ny=BJ_ny$, where $J_n$'s are the resolvents of $B$.  We construct an infinite sequence of almost-linear (i.e., Lipschitz) problems (see \eqref{wn}), whose unique and strong solutions, say $y_n$, can be easily obtained via the classical semigroup theory.
  \item[Step 2:] We obtain a \emph{unique continuation property} (Lemma \ref{Theo. 2.2}) which is valid for any weak solution of the approximate solution model that vanishes on $\omega$.  It is noteworthy to mention that the unique continuation property is not stated for a linear model, but rather given for the approximate solution model whose nonlinear part is globally Lipschitz in $L^2(\Omega)$. This allows us to simplify the proof of an important inequality (see Lemma \ref{lema1}). Uniqueness of solution for the approximate model is critical in the proof of the unique continuation.
  \item[Step 3:] By using the multipliers we show that the approximate solutions $y_n$'s are nonincreasing at the $L^2-$level, and moreover uniformly bounded in $n$ at the $H^1$-level. The assumption \eqref{anabla} on the damping coefficient plays a critical role in controlling the $H^1$-norm.
  \item[Step 4:] The exponential decay of approximate solutions is reduced to proving the inequality given in \eqref{estUCP}.  This is proven by contradiction utilizing the unique continuation property given in Step 2.
  \item[Step 5:] As a last step, we use the classical compactness arguments based on the uniform bounds of the approximate solutions in suitable spaces to pass to a subsequence which converges to a soughtafter weak solution of the original model.  The decay of this weak solution is obtained via weak lower semicontinuity of the norm.
  \item[Step 6:] We extend the proof of Theorem \ref{theorem 2.1} to unbounded domains in the particular cases where  $ \Omega $ is either the whole space or an exterior domain.
    \item[Step 7:] We finish the paper with a numerical section, based on a Finite Volume Method, where illustrations verify the proved decay rate.
\end{itemize}
	 	 \section{Approximate solutions, weak solution, unique continuation, stabilization}
	 	 \setcounter{equation}{0}
	 	 This section is devoted to the proof of the main result when $ \Omega $ is a bounded domain.  Monotone operator theory is used as in Özsar\i\,et al. \cite[Section 4]{Ozsar} to construct approximate solutions, except that the treatment here also includes the case of a space dependent damping coefficient.  Once such solutions are constructed we prove that they obey a mass decay law at the $L^2$ level via a unique continuation property.  Finally, we pass to the limit to construct a weak solution.  A similar mass decay for this weak solution is obtained via weak lower semicontinuity argument.	 	
	 	
We start our construction of approximate solutions for problem (\ref{problema}) by replacing the nonlinear source with its Yosida approximations.   To this end, we consider the nonlinear operator $\mathcal{B}$   on $L^2\left(\Omega\right)$ defined by
	 	 \begin{eqnarray}\label{yosida}D(\mathcal{B})&\doteq&\set{y\in\,L^2\left(\Omega\right); \,|y|^p\,y\in\,L^2\left(\Omega\right)},\\
	 	 \mathcal{B} y&\doteq&|y|^p\,y,\,\,\forall\,y\in\,D(\mathcal{B})\,.\label{B}
	 	 \end{eqnarray} It is well known that $\mathcal{B}$ is m-accretive (see e.g., Okazawa and Yokota \cite[Lemma 3.1]{Okazawa}). 	
	 	 Thus,  we can define the (Lipschitz) Yosida approximations $\mathcal{B}_{n}$ of $\mathcal{B}$ in terms of the resolvents $J_{n}$:
	 	 \begin{eqnarray}\label{Jrn}J_{n}\doteq\left(1+\frac1n\,\mathcal{B}\right)^{-1}\end{eqnarray} and \begin{eqnarray}\label{Bn}
	 	 \mathcal{B}_{n}\doteq n\,(I-J_{n})=\mathcal{B}\,J_{n}\,.\end{eqnarray}
	 	 One can represent the operators $\mathcal{B}$ and $\mathcal{B}_{n}$
	 	 as subdifferentials  $$\mathcal{B}=\partial\,\psi\quad\hbox{ and }\quad
	 	 \mathcal{B}_{n}=\partial\,\psi_{n},$$ where $\psi$ and $\psi_{n}$ are given by
	 	 \begin{eqnarray}\label{psi}\psi(y)\doteq\left\{\begin{aligned}&\frac1{p+2}\,||y||_{L^{p+2}\left(\Omega\right)}^{p+2}\quad\hbox{ for
	 	 }y\in\,L^{p+2}\left(\Omega\right)\\
	 	 &\infty\quad\hbox{ otherwise }\end{aligned}\right.\end{eqnarray}
	 	 and
	 	 \bq\label{identpsi}\psi_{n}(y)\doteq\min_{v\in\,L^2\left(\Omega\right)}\left\{\frac
	 	 n2\,||v-y||_{L^2\left(\Omega\right)}^2+\psi(v)\right\}\\\nonumber&=&\frac1{2n}\,\norm{\mathcal{B}_{n}\,y
	 	 }{L^2\left(\Omega\right)}^2+\psi(J_{n}\,y),\,\,y\in\,L^2\left(\Omega\right).\eq	 		 	
	 	 Moreover, one has
	 	 \begin{equation}
	 	 \label{ineqs}\psi(J_n(z))\leq \psi_n(z)\leq \psi(z)\,.
	 	 \end{equation} 	
	 	 Now, given $ y_0\in\,\mathcal{X}, $ we choose a sequence of elements $ \{y_{n, 0}\} \subset \mathcal{X}\cap H^2(\Om)$ such that $y_{n, 0}\rightarrow y_0$ in   $\mathcal{X}.$
	 	We first consider the following approximate problems:
	 	 \begin{eqnarray}\label{wn}
	 	 \left\{\begin{aligned}&i\,\partial_t\,y_n+\Delta\,y_n-F_{n}(y_n)=0
	 	 \quad\hbox{ in }\quad \Omega\,\times\,(0,T),\\
	 	 &y_n(0)=y_{n,0}\quad\hbox{ in }\quad \Omega,\end{aligned}\right. \end{eqnarray} where $F_{n}(y_n) :=\mathcal{B}_{n}(y_n)-ia(x)\,y_n$.
	 	
	 	 As $\mathcal{B}_{n}$ is Lipschitz with say Lipschitz constant $L_{n}$, we deduce that $ F_{n} $ is also Lipschitz. Indeed, let $ y,z\in\,L^2\left(\Omega\right)$, then
	 	 \begin{eqnarray*}
	 	 	\norm{F_{n}(y)-F_{n}(z)}{L^2\left(\Omega\right)}&\leq&\norm{\mathcal{B}_{n}(y)-\mathcal{B}_{n}(z)}{L^2\left(\Omega\right)}+\norm{a(\cdot)(y-z)}{L^2\left(\Omega\right)}\\
	 	 	&\leq&\left(\,L_{n}+\norm{a}{L^{\infty}\left(\Omega\right)}\right)\,\norm{y-z}{L^2\left(\Omega\right)}\,.
	 	 \end{eqnarray*} By using the standard semigroup theory \cite{Pazy}, we obtain a unique solution
	 	 $y_n$ which solves \eqref{wn} and satisfies $y_n\in\,C([0,\infty);H_0^1(\Om)\,\cap\,H^2\left(\Omega\right))\,\cap\,C^1([0,\infty);L^2\left(\Omega\right))$.

Next, we prove the following unique continuation result for the approximate solutions:
	 	 \begin{lemma}[Unique Continuation]\label{Theo. 2.2} Let $n\ge 1$ be fixed and
	 	 	 $u\in\,L^{\infty}(0,T;H_0^1(\Omega))\,\cap\,C([0,T];L^2(\Omega))$ be a weak solution of
\begin{equation}
	\label{probF2u}
	\begin{cases}
	\,i\partial_t u+\Delta u=F_{n}u\quad \hbox{ in } \Omega\,\times (0,T)\\
	u=0 \quad\text{ a.e. in }\omega\,\times\,(0,T); \end{cases}
	\end{equation} then $u\equiv 0$ on $\Omega \times (0,T)$.
	 	 \end{lemma}	 	
\begin{proof}[Proof of Lemma \ref{Theo. 2.2}]
In order to prove this theorem, we will use the unique continuation principle presented by \cite{las03}.  The unique continuation argument of \cite{las03} does not directly apply to the problem under consideration here, and there are some technical challenges related with smoothness of solutions and the source function.  In \cite{las03}, unique continuation was proved for $H^{2,2}(\Omega\times (0,T))$ solutions assuming $F_nu$ can be written as $q_0(x,t)u$ for some $q_0\in L^\infty(\Omega)$, or for energy solutions assuming $q_0$ satisfies further rather strong smoothness conditions.  Although we can put $F_nu$ in the form $q_0u$ by simply defining \[ q_0(x,t) := \begin{cases}
          \frac{(F_nu)(x,t)}{u(x,t)} & \text{ if } u(x,t)\neq 0 \\
          0 &  \text{ if } u(x,t)= 0,
       \end{cases}
    \] one cannot use the unique continuation theory at the $H^{2,2}(\Omega\times (0,T))$ level because the solutions of \eqref{wn} only belong to $C([0,T];H^2(\Omega))\cap H^1(0,T;L^2(\Omega))$, which is rougher.  Similarly, we are also not in a position to use the unique continuation at the energy level because $q_0$ does not satisfy the extra conditions given in \cite{las03} for energy solutions.  In order to deal with this difficulty, we will utilize the\emph{ uniqueness} of weak solution to \eqref{wn} together with a compactness argument. Uniqueness is unknown for the NLS with power nonlinearity on high dimensional domains, but luckily we know it is true for the approximate model \eqref{wn} with Lipschitz nonlinearity.  This is another advantage of using Yosida approximations here.

We start by shifting the topology up by constructing (sufficiently smooth) \emph{approximations of approximations}. To this end, for a given $n$, let us consider the problem:
\begin{equation}
	\label{probF2}
	\begin{cases}
	\,i\partial_t w_m+\Delta w_m=f_m(x,t)\quad \hbox{ in } \Omega\,\times (0,T)\\
	w_m=0 \quad\text{ a.e. in }\omega\,\times\,(0,T), \end{cases}
	\end{equation} together with $w_m(0)=w_{m0}\in H^4(\Omega)\cap H_0^1(\Omega)$, where $\displaystyle \lim_{m\rightarrow\infty} w_{m0}= u(0)$ in $H_0^1(\Omega)$, and $f_m\in L^2(0,T;H^2(\Omega))\cap H^1(0,T;L^2(\Omega))$ s.t.  $\displaystyle\lim_{m\rightarrow\infty} f_m = F_nu$ in $L^2(Q)$.
By the linear theory of the Schrödinger equation, \eqref{probF2} has a solution $w_m\in L^2(0,T;H^4(\Omega))\cap H^2(0,T;L^2(\Omega))$. Therefore, in particular $w^m\in H^{2,2}(\Omega\times(0,T))$ and it also satisfies the conditions given in \cite{las03}[2.1.1 (b)].  Note that the right hand side of \eqref{probF2} is simply $$\overrightarrow{0}\cdot \nabla w_m+0\cdot w_m +f_m$$ with respect to the notation given in \cite{las03}. Due to the unique continuation principle \cite{las03}[Cor 2.1.2-ii], we deduce that $w_m\equiv 0$.

By using the multipliers on \eqref{probF2} and compactness arguments we can extract a subsequence of $w_m$ which converges to a weak solution $w$ of \eqref{wn}.  But then $w(0)=u(0)$, and $w$ and $u$ solve the same equation in the weak sense. But $w$ cannot be anything other than zero since all $w_n$ were zero. On the other hand, the weak solution of \eqref{wn} is unique, and therefore we must have $0\equiv w\equiv u$.
\end{proof} 	
	 	 Now, taking the $L^2$-inner product of (\ref{wn}) with $y_n$ and looking at the imaginary parts, we see that
	 	 \begin{equation}
	 	 \Re(\partial_t\,y_n, y_n)_{L^2(\Omega)}-\underbrace{\Im(\nabla y_n,\nabla y_n)_{L^2(\Omega)}}_{=0}-\underbrace{\Im(\mathcal{B}_n(y_n),y_n)_{L^2(\Omega)}}_{=0}+(a(x)\,y_n,y_n)_{L^2(\Omega)}=0,\label{ident_L2}
	 	 \end{equation}
	 	 where the third term vanishes, since by \eqref{Bn} we have
	 	 \begin{equation}
	 	 \begin{aligned}
	 	 \label{Bnyn0}(\mathcal{B}_n(y_n),y_n)_{L^2(\Omega)}&=\left(\mathcal{B}_n(y_n),\frac1n\mathcal{B}_n(y_n)+J_n(y_n)\right)_{L^2(\Omega)}\\
	 	 &=\frac1n\,\|\mathcal{B}_n(y_n)\|_{L^2(\Omega)}^2+\left(\mathcal{B}_n(y_n),J_n(y_n)\right)_{L^2(\Omega)}\\
	 	 &=\frac1n\,\|\mathcal{B}_n(y_n)\|_{L^2(\Omega)}^2+\|J_n(y_n)\|_{L^{p+2}(\Omega)}^{p+2}\,.
	 	 \end{aligned}
	 	 \end{equation}
	 	 Hence, we obtain
	 	 \begin{eqnarray}\label{ynL2}
	 	 \label{Ld}\frac12\,\frac{d}{dt}\norm{y_n}{L^2(\Omega)}^2=-\int_{\Omega} a(x)|y_n|^2\,dx \leq \,0.
	 	 \end{eqnarray}

%

        	(\ref{ynL2}) implies that the mass $ E_{n,0}(t):=\frac12\,\|y_n(t)\|_{L^2(\Om)}^2 $ is non-increasing. Integrating \eqref{ynL2} on $(0,T)$, we obtain
	 	 	\begin{equation}
	 	 	\label{ynL2_1}E_{n,0}(T)+\int_0^T\int_{\Om}a(x)|y_n|^2\,dx\,dt=E_{n,0}(0),
	 	 	\end{equation}
	 	 	and from the assumption $ a(x)\geq a_0 >0 $ a.e. on $ \omega,  $ we get
	 	 	\begin{equation}
	 	 	\begin{aligned}
	 	 	\label{ynL2_2}a_0\int_0^T \int_{\omega}|y_n|^2\,dx\,dt&\leq \int_0^T \int_{\Omega}a(x)\,|y_n|^2\,dx\,dt\\&= E_{n,0}(0)-E_{n,0}(T)=-\frac{1}{2}\left[\int_{\Omega} |y_n|^2\,dx\right]_0^T,
	 	 	\end{aligned}
	 	 	\end{equation}
	 	 	and thus,
	 	 	\begin{equation}
	 	 	\begin{aligned}
	 	 	\label{ynL2_3}\int_0^T \int_{\omega}|y_n|^2\,dx\,dt\leq-\frac{1}{2a_0} \left[\int_{\Omega} |y_n|^2\,dx\right]_0^T.
	 	 	\end{aligned}\end{equation}
	 	 	
	 	 	Therefore,  by (\ref{ynL2_3}), we have the following estimate:
	 	 	\begin{equation}
	 	 	\begin{aligned}\label{energyest}
	 	 	\int_{0}^{T}E_{n,0}(t)dt&=\frac{1}{2}\int_{0}^{T}\int_{\omega}|y_n|^2 dx\, dt+\frac{1}{2}\int_{0}^{T}\int_{\Om\setminus \omega}|y_n|^2 dx\,dt\\&\leq -\frac{1}{2a_0}
	 	 	\left[\int_{\Om}|y_n|^2 dx\right]_{0}^{T}+\underbrace{\frac{1}{2}\int_{0}^{T}\int_{\Om\backslash\om}
	 	 	|y_n|^2dx\,dt}_{I_n}.
	 	 	\end{aligned}
	 	 	\end{equation}
	
 	We will prove in Lemma \ref{lema1} below a useful inequality for the integral $I_n$. Before proving this lemma, let us make a few more observations about the approximate solutions.

  Multiplying \eqref{wn} by $-i$ and rearranging the terms we get
	 	 $$\partial_t\,y_n=i\,\Delta\,y_n-i\,\mathcal{B}_{n}(y_n)-\,a(x)\,y_n\,.$$
 	
	 	 From the above identity, it follows that
	 	 \begin{equation}
	 	 \label{ident_1}\begin{aligned}
	 	 \Re\left(-\Delta\,y_n+\mathcal{B}_n\,(y_n), \partial_t\,y_n\right)_{L^2(\Omega)}&=\Re\left(-\Delta\,y_n+\mathcal{B}_n\,(y_n), i\, \Delta y_n - i\,\mathcal{B}_n\,(y_n) - a(x) y_n\right)_{L^2(\Omega)}\\
	 	 &=\cancelto{0}{\Re\,i\|\Delta\,y_n(t)\|_{L^2(\Omega)}^2}+\Re\left(\Delta\,y_n, i\,\mathcal{B}_n\,(y_n)\right)_{L^2(\Omega)}\\
	 	 &+\Re\left(\Delta\,y_n, a(x) y_n\right)_{L^2(\Omega)}
	 	 +\Re\left(\mathcal{B}_n\,(y_n), i\, \Delta y_n \right)_{L^2(\Omega)}\\
	 	 &+\cancelto{0}{\Re\,i\|\mathcal{B}_n(y_n(t))\|_{L^2(\Omega)}^2}-\Re(\mathcal{B}_n(y_n),a(x)\,y_n)_{L^2(\Omega)}\,.
	 	 \end{aligned}
	 	 \end{equation}
%
%
Taking into account \begin{equation*}
	 	 \label{identity_lemma}\Re\left(\Delta\,y_n, i\,\mathcal{B}_n\,(y_n)\right)_{L^2(\Omega)}
	 	 +\Re\left(\mathcal{B}_n\,(y_n), i\, \Delta y_n \right)_{L^2(\Omega)}=0,
	 	 \end{equation*}
	 	 from \eqref{ident_1}, we obtain
	 	 \begin{equation}
	 	 \label{ident_2}\begin{aligned}
	 	 \Re\left(-\Delta\,y_n+\mathcal{B}_n\,(y_n), \partial_t\,y_n\right)_{L^2(\Omega)}&=\Re\left(\Delta\,y_n, a(x) y_n\right)_{L^2(\Omega)}-\Re(\mathcal{B}_n(y_n),a(x)\,y_n)_{L^2(\Omega)}\\
	 	 &=-\Re\left(\nabla\,y_n, \nabla\,a(x) y_n\right)_{L^2(\Omega)}-\int_\Omega a(x)|\nabla\,y_n|^2\,dx \\
	 	 &-\Re(\mathcal{B}_n(y_n),a(x)\,y_n)_{L^2(\Omega)}\,.
	 	 \end{aligned}
	 	 \end{equation}
	 	 It follows from Showalter \cite[Chapter IV, Lemma 4.3]{Show} that
	 	 \begin{eqnarray}\label{show}&\displaystyle\frac{d}{dt}\psi_{n}(y_n)=Re\left(\mathcal{B}_{n}
	 	 (y_n),\partial_t\,y_n\right)_{L^2\left((\Omega)\right)}.\end{eqnarray}
	 	 Using \eqref{show}, we get
	 	 \begin{equation}
	 	 \label{ident_3}\begin{aligned}
	 	 \Re\left(-\Delta\,y_n+\mathcal{B}_n\,(y_n), \partial_t\,y_n\right)_{L^2(\Omega)}&=\frac{d}{dt}\left[\frac12\,\|\nabla\,y_n(t)\|^2_{L^2(\Omega)}+\psi_{n}(y_n)\right]\,.
	 	 \end{aligned}
	 	 \end{equation}
 	 Combining \eqref{ident_2} and \eqref{ident_3}, it follows that
	 	 \begin{equation}
	 	 \label{ident_4}\begin{aligned}
	 	 &\frac{d}{dt}\left[\frac12\,\|y_n(t)\|^2_{H_0^1(\Omega)}+\psi_{n}(y_n)\right]+\int_\Omega a(x)\,|\nabla\,y_n|^2\,dx\\
	 	 &= -\Re\left(\nabla\,y_n, \nabla\,a(x) y_n\right)_{L^2(\Omega)}-\Re(\mathcal{B}_n(y_n),a(x)\,y_n)_{L^2(\Omega)}\,.
	 	 \end{aligned}
	 	 \end{equation}
	 	 Using (\ref{Bn}), we obtain
	 	 {\small\begin{eqnarray}\label{bridge_Bnyn0} \hspace{0.8cm} -\Re(\mathcal{B}_n(y_n),a(x)\,y_n)_{L^2(\Omega)}=-\frac1n\,\int_{\Omega}a(x)|\mathcal{B}_n(y_n)|^2\,dx-\int_{\Omega}a(x)|J_n(y_n)|^{p+2}\,dx\leq\,0\,.
	 	 	\end{eqnarray}}
	 	 From \eqref{ident_4} and \eqref{bridge_Bnyn0} and taking into account the assumption \eqref{anabla}, we have
	 	 \begin{equation}
	 	 \label{ident_5'}\begin{aligned}
	 	 &\frac{d}{dt}\left[\frac12\,\|y_n(t)\|^2_{H_0^1(\Omega)}+\psi_{n}(y_n)\right]+\int_\Omega a(x)\,|\nabla\,y_n|^2\,dx\\
	 	 &= -\Re\left(\nabla\,y_n, \nabla\,a(x) y_n\right)_{L^2(\Omega)}\underbrace{-\Re(\mathcal{B}_n(y_n),a(x)\,y_n)_{L^2(\Omega)}}_{\leq\,0}\\
	 	 &\leq \int_\Omega |\nabla\,a(x)|\,|y_n|\,|\nabla\,y_n|\,dx\leq {C}\,\int_\Omega |a(x)|^\frac{1}{2}\,|y_n|\,|\nabla\,y_n|\,dx\,.
	 	 \end{aligned}
	 	 \end{equation}
	 	
	 	 Employing the inequality $ ab\leq\, \frac1{4\varepsilon}a^2+\varepsilon\,b^2 $ ($\epsilon>0$)   we obtain
	 	 \begin{equation}
	 	 \label{ident_6}\begin{aligned}
	 	 &\frac{d}{dt}\left[\frac12\,\|y_n(t)\|^2_{H_0^1(\Omega)}+\psi_{n}(y_n)\right]+\int_\Omega a(x)\,|\nabla\,y_n|^2\,dx\\
	 	 &\leq \frac{{C}^2}{4\,\varepsilon}\,
	 	 \int_\Omega \,|y_n|^2\,dx+{\varepsilon}\,\int_\Omega a(x)\,|\nabla\,y_n|^2\,dx.
	 	 \end{aligned}
	 	 \end{equation}
	 	 that is,
	 	 \begin{equation}
	 	 \label{ident_7}\begin{aligned}
	 	 &\frac{d}{dt}\left[\frac12\,\|y_n(t)\|^2_{H_0^1(\Omega)}+\psi_{n}(y_n)\right]+(1-\varepsilon)\,\int_\Omega a(x)\,|\nabla\,y_n|^2\,dx\\&\leq \frac{{C}^2}{4\,\varepsilon}\int_\Omega \,|y_n|^2\,dx\,.
	 	 \end{aligned}
	 	 \end{equation}
	 	
	 	 Considering $ \varepsilon>0 $ small enough, we conclude that
	 	 \begin{equation}
	 	 \label{ident_8}\begin{aligned}
	 	 &\frac{d}{dt}\left[\frac12\,\|y_n(t)\|^2_{H_0^1(\Omega)}+\psi_{n}(y_n)\right]\leq \frac{{C}^2}{4\,\varepsilon}\int_\Omega \,|y_n|^2\,dx\,.
	 	 \end{aligned}
	 	 \end{equation}
	 	 Integrating \eqref{ident_8} in variable $ t\in\,[0,T]$, we get
	 	
	 	 \begin{equation}
	 	 \label{ident_9}\frac{1}{2}\|y_n(t)\|^2_{H_0^1(\Omega)}+\psi_{n}(y_n) \leq \frac{1}{2}\|y_{n,0}\|_{H_0^1(\Omega)}^2 +\psi_{n}(y_{n,0})+\frac{{C}^2}{4\,\varepsilon}\,\int_0^T\int_\Omega \,|y_n|^2\,dx\,dt.
	 	 \end{equation}
	 	
	 	 	\begin{lemma}There exists some $n_0\ge 1$, such that for any fixed $n\ge n_0$, the corresponding solution $y_n$ of \eqref{wn} will satisfy the inequality
	 	 		\begin{equation}
	 	 		\int_{0}^{T}\int_{\Om\backslash \om}|y_n|^2dxdt\leq
	 	 		c\int_0^T\int_{\Om}a(x)|y_n|^2\,dx\,dt
	 	 		\label{estUCP}\end{equation} for some $c$ (which depends on $\|y_0\|_{H_0^1(\Omega)}$).
	 	 		
	 	 		\label{lema1}
	 	 	\end{lemma}

	 	 	\begin{proof}[Proof of Lemma \ref{lema1}] The initial datum $y_0\in H_0^1(\Omega)$ in the original model \eqref{problema} is either zero (case (i)) or not zero (case (ii)).

  In the first case, namely if $y_0\equiv 0$, then we can simply set $y_{n,0}\equiv 0$ for $n\ge 1$, which will trivially converge to $y_0\equiv 0$ in $H_0^1(\Omega)$, and the corresponding unique solution of \eqref{wn} will be $y_n\equiv 0$.  Therefore, \eqref{estUCP} will readily hold.

  In the second case, where $y_0\not\equiv 0$, we can choose two strictly positive numbers $\ell,L>0$ such that \begin{equation}\label{L} 0<\ell< \|y_0\|_{L^2(\Omega)} \text{ and } \|y_0\|_{H_0^1(\Omega)}<L,\end{equation} say for instance $\ell\doteq\frac{1}{2}\|y_0\|_{L^2(\Omega)}>0$, and $L\doteq2\|y_0\|_{H_0^1(\Omega)}>0$. On the other hand, we know that $y_{n,0}$ are chosen to strongly converge to $y_0$ in $H_0^1(\Omega)$.  Therefore, there exists $n_0>0$ such that for all $n\ge n_0$, $y_n$ will satisfy
  \begin{equation}\label{Lyn} 0<\ell< \|y_{n,0}\|_{L^2(\Omega)}\le\|y_0\|_{L^2(\Omega)}  \text{ and }  \|y_0\|_{H_0^1(\Omega)}\le \|y_{n,0}\|_{H_0^1(\Omega)}<L.\end{equation}

 Now, we claim that under the condition \eqref{Lyn} on $y_{n,0}$, the solution $y_n$ of \eqref{wn} satisfies \eqref{estUCP}.  In order to prove the claim, we argue by contradiction.  Now, if the claim is false, then no matter what we choose for the constant $c$ in \eqref{estUCP}, we can find an initial datum for problem  \eqref{wn} whose corresponding solution violates  $\eqref{estUCP}$. For example if $c=k\ge 1$, then there exists an initial datum, say $y_{n,0}^k\in H^2(\Omega)\cap H_0^1(\Omega)$, satisfying the properties
 \begin{equation}\label{Lyn2} 0<\ell< \|y_{n,0}^k\|_{L^2(\Omega)}\le\|y_0\|_{L^2(\Omega)}  \text{ and }  \|y_0\|_{H_0^1(\Omega)}\le \|y_{n,0}^k\|_{H_0^1(\Omega)}<L,\end{equation}
  but whose corresponding solution, say $y_n^k$, violates $\eqref{estUCP}$ in the sense
 \begin{equation}
	 	 		\int_{0}^{T}\int_{\Om\backslash \om}|y_n^k|^2dxdt>
	 	 		k\int_0^T\int_{\Om}a(x)|y_n^k|^2\,dx\,dt.
	 	 		\label{estUCP2}\end{equation}
Moreover, we can say this for each $k\ge 1$, and hence obtain a sequence of initial data $\{y_{n,0}^k\}_{k=1}^\infty$, each of whose elements satisfy $\eqref{Lyn2}$ and a sequence of corresponding solutions 	$\{y_{n}^k\}_{k=1}^\infty$, each of whose elements solves \eqref{wn} but also satisfies \eqref{estUCP2}.

Since $y_n^k$ is bounded in $L^\infty(0,T;H_0^1(\Omega))$, we obtain a subsequence of $y_n^k$ (denoted same) which converges (weakly-$^*$) to some $u$ in $L^\infty(0,T;H_0^1(\Omega))$. Moreover, $F_n(y_n^k)$ is bounded in $L^\infty(0,T;L^{(p+2)'}(\Omega))$; therefore there is some $\chi$ such that $F_n(y_n^k)$ (indeed a subsequence of it) weakly-$^*$ converges to in $L^\infty(0,T;L^{(p+2)'}(\Omega))$. It follows that $\partial_ty_n^k$ is bounded in $L^\infty(0,T;H^{-1}(\Omega))$ and (a subsequence of) it  weakly-$^*$ converges to $u_t$ in $L^\infty(0,T;H^{-1}(\Omega))$. By compactness, we have $y_n^k$ converges strongly to $u$ in $L^\infty(0,T;L^2(\Omega))$ and a.e. on $[0,T]\times \Omega$.  Then, we have $\chi=F_n(u)$, and $u$ satisfies the main equation of the approximate model \eqref{wn}. Moreover, since the left hand side of \eqref{estUCP2} is bounded, we have
\begin{equation}\label{lemmaimp3}
  \int_0^T\int_\Omega a(x)|y_n^k|^2dxdt \rightarrow 0.
\end{equation} Therefore, using the assumption $a(x)\ge a_0>0$ on $\omega$, we have
\begin{equation}\label{lemmaimp4}
  \int_0^T\int_\omega |y_n^k|^2dxdt \rightarrow 0,
\end{equation} which implies that $u\equiv 0$ on $\omega$ since $y_n^k$ strongly converges to $u$ in $L^\infty(0,T;L^2(\Omega))$.  Therefore, $u$ must be zero by unique continuation property.  But then we define
\begin{equation}
	 	 		\nu_k=\sqrt{\int_0^T\int_{\Omega-\omega}|y_n^k|^2dxdt} \label{normalizeduk}
	 	 		\end{equation} together with $v_k=y_n^k/\nu_k$. Dividing both sides of \eqref{estUCP2} by $\nu_k^2$, we obtain in the same way
\begin{equation}\label{lemmaimp3a}
  \int_0^T\int_\Omega a(x)|v_k|^2dxdt \rightarrow 0,
\end{equation} which implies
\begin{equation}\label{lemmaimp4a}
  \int_0^T\int_\omega |v_k|^2dxdt \rightarrow 0.
\end{equation} But we also  know that $y_n^k\rightarrow u\equiv 0$ in $L^2(0,T;L^2(\Omega))$, and hence $\nu_k\rightarrow 0$. We in particular have
 \begin{equation}\label{important}
   \|v_k(0)\|_{L^2(\Omega)} = \frac{\|y_n^k(0)\|_{L^2(\Omega)}}{\nu_k} \ge \frac{\ell}{\nu_k} \rightarrow \infty\,\, (\text{as }k\rightarrow \infty).
 \end{equation}
	 	 		On the other hand,  the energy dissipation law  yields
	 	 		\bq\label{est2}E_{n,0}^k(T_0)=E_{n,0}^k(0)-\int_0^{T_0}\int_\Om a(x)| y_n^k(x,t)|^2\,dx\,dt\,.\eq
	 	 		
	 	 		Combining (\ref{energyest}) and (\ref{est2}), we infer
	 	 		\begin{eqnarray*}
	 	 			\int_0^T
	 	 			E_{n,0}^k(t)\,dt&\leq& \frac{1}{a_0}
	 	 			E_{n,0}^k(0)+\frac{1}{2}\int_{0}^{{T_0}}\int_{\Om\backslash\om}
	 	 			|y_n^k|^2\,dx\,dt\\&\leq&\frac{1}{a_0} \left[E_{n,0}^k({T_0})+\int_0^T\int_\Om a(x)| y_n^k|^2\,dx\,dt\right]\\
	 	 			&+&\frac{1}{2}\int_{0}^{{T_0}}\int_{\Om\backslash\om}\,|y_n^k|^2\,dx\,dt\,.
	 	 		\end{eqnarray*}
	 	 		
	 	 		Since $E_{n,0}^k(t)$ is a non-increasing function, one has from the above inequality the following estimate:
	 	 		\bq\label{est3}
	 	 		E_{n,0}^k({T_0})\left({T_0}-\frac{1}{a_0} \right)&\leq&\frac{1}{a_0} \int_0^{T_0}\int_\Om a(x)| y_n^k|^2\,dx\,dt\\ \nonumber&+&\frac{1}{2}\int_{0}^{{T_0}}\int_{\Om\backslash\om}|y_n^k|^2\,dx\,dt\, .\eq
	 	 		
	 	 		From (\ref{est3}), we deduce that, for sufficiently large ${T_0}>0$, there exists
	 	 		$C=C(a_0,T_0)$ verifying
	 	 		\bq\label{est4}E_{n,0}^k({T_0})\leq\,C\,\left[\int_0^{T_0}\int_\Om a(x)| y_n^k|^2\,dx\,dt+\int_{0}^{{T_0}}\int_{\Om\backslash\om}|y_n^k|^2dx\,dt\right]\!.\eq
	 	 		
	 	 		Combining (\ref{est2}) and (\ref{est4}) we finally deduce that
	 	 		\bq\label{est5}E_{n,0}^k(0)\leq\,\hat{C}\,\left(\int_0^{T_0}\int_\Om a(x)| y_n^k|^2\,dx\,dt+\int_{0}^{{T_0}}\int_{\Om\backslash\om}|y_n^k|^2\,dx\,dt\right)\,.\eq
	 	 		
	 	 		From (\ref{est5}) we infer for any $k\in\,\mathbb{N}$, that
	 	 		\bq\label{est6}\frac{E_{0,n}^k(0)}{\nu_k^2}\leq\,\hat{C}\,\left(\int_0^{T_0}\int_\Om a(x)| v_k|^2\,dx\,dt+1\right)\,.\eq
	 	 		
	 	 		Thus, we guarantee the existence of $M> 0$ such
	 	 		that
	 	 		\bq\label{est7}\frac{1}{2}\norm{v_k(0)}{L^2(\Omega)}^2=\frac{\norm{y_{n,0}^k}{L^2(\Omega)}^2}{2\nu_k^2}=\frac{E_{0,n}^k(0)
	 	 		} { \nu_k^2 }\leq M\quad\hbox{ for all } k\in\,\mathbb{N},\eq which establishes a bound
	 	 		for the initial data $v_k(0)$ in $L^2$-norm. This contradicts with \eqref{important}.  Hence, by contradiction, $y_n$ must satisfy \eqref{estUCP}.\\
	 	 		
	 	 	\end{proof}

	 	 	We notice that the equations \eqref{est2}-\eqref{est5} are all valid for $y_n^k$ replaced by $y_n$, too.
	 	 	It follows from \eqref{ynL2_1} that
	 	 	\begin{equation}
	 	 	\label{final_inequality}
	 	 	E_{n,0}(T_0)\leq E_{n,0}(0) \leq C\,\int_0^{T_0}\int_{\Om}a(x)|y_n|^2\,dx\,dt,
	 	 	\end{equation} where $C$ is a positive constant.
	 	 	
	 	 	Now, combining \eqref{ynL2_1} and \eqref{est5}, and using the inequality \eqref{estUCP} given in the above lemma, we obtain
	 	 	\begin{equation}
	 	 	\begin{aligned}
	 	 	E_{n,0}(T_0)&\leq C\,\int_0^{T_0}\int_{\Om}a(x)|y_n|^2\,dx\,dt\\&=C\,\left(E_{n,0}(0)-E_{n,0}(T_0)\right)\,.
	 	 	\end{aligned}
	 	 	\end{equation}
	 	 	Therefore,\begin{eqnarray}
	 	 	\label{inequality}E_{n,0}(T_0)&\leq&\left(\frac{C}{1+C}\right)\,E_{n,0}(0).
	 	 	\end{eqnarray}
	 	 	Repeating the procedure for $nT_0$, $n\in \mathbb{N}$, we deduce
	 	 	$$E_{n,0}(nT_0)\leq\frac{1}{(1+\hat{C})^n}E_{n,0}(0)$$ for all $n\geq 1.$\\
	 	 	
	 	 	Let us consider, now,  $t\geq T_0$, and then write $t=nT_0+r,$ $0\leq r<T_0.$
	 	 	Thus, $$E_{n,0}(t)\leq E_{n,0}(t-r)=E_{n,0}(nT_0)\leq
	 	 	\frac{1}{(1+\hat{C})^n}\,E_{n,0}(0)=\frac{1}{(1+\hat{C})^{\frac{t-r}{T_0}}}E_{n,0}(0).
	 	 	$$
	 	 	
	 	 	Setting $\displaystyle C_0=\textrm{e}^{\frac{r}{T_0}\ln(1+\hat{C})}$ and
	 	 	$\lambda_0=\frac{\ln(1+\hat{C})}{T_0}>0,$ we obtain
	 	 	\begin{equation}\label{exp}
	 	 	E_{n,0}(t)\leq C_0 \,\textrm{e}^{- \lambda_0 t}E_{n,0}(0); \ \ \forall t\geq
	 	 	T_0,
	 	 	\end{equation} which proves the exponential decay  to problem (\ref{wn}).

 Combining \eqref{ineqs}, \eqref{exp}, and \eqref{ident_9}, it follows that
	 	 \begin{equation}
	 	 \label{bound_H1}\|y_n(t)\|^2_{H_0^1(\Omega)}+\psi(J_{n}(y_n)) \leq \frac{C_0C}{\lambda_0\epsilon^2}\|y_{n0}\|_{\mathcal{X}}\,.
	 	 \end{equation}

	 	 The inequality (\ref{bound_H1}) and the boundedness of the sequence $\{y_{n,0}\}$ in ${\mathcal{X}}$ enable us to conclude that
	 	 \begin{eqnarray}
	 	 &&\label{bounded_yn}\{y_n\}\hspace{1.1cm}\hbox{ is bounded in } \quad L^{\infty}(0,T; H_0^1(\Omega))\\
	 	 &&\label{bounded_J_{p,n}}\{J_{n}(y_n)\}\quad
	 	 \hbox{ is bounded in }L^\infty(0,T; L^{(p+2)}(\Omega))\hookrightarrow L^{(p+2)}(0,T; L^{(p+2)}(\Omega)).
	 	 \end{eqnarray}
	 	 Notice that
	 	 \bq\label{BnJn}\mathcal{B}_n(y_n)=\mathcal{B}(J_n(y_n))=|\,J_n(y_n)\,
	 	 |^p\, J_n( y_n)\,
	 	 .\eq So,  from (\ref{bounded_J_{p,n}}) 
	 	 and (\ref{BnJn}),  we get
	 	 \bq &&\label{bounded_B_{p,n}}\{\mathcal{B}_{n}\,y_n\}\quad\hbox{ is bounded in
	 	 }\quad L^{(p+2)'}(0,T; L^{(p+2)'}(\Omega )).
	 	 \eq
	 	 On the other hand, by (\ref{bounded_yn})  and  (\ref{bounded_B_{p,n}})  we
	 	 observe
	 	 \begin{eqnarray*}\norm{\partial_t\,y_n}{\mathcal{X}'}
	 	 	&=&\sup_{\|{\varphi}\|_{ \mathcal{X}}=1}\,
	 	 	\left\{\left({\partial_t\,y_n},{\varphi}\right)_{L^2(\Omega)}\right\}\\
	 	 	&=&\sup_{\|{\varphi}\|_{\mathcal{X}}=1}\{({i\,\Delta\,y_n},
	 	 	\varphi)_{L^2(\Omega)}-\left({i\,\mathcal{B}_{n}(y_n)},{\varphi}\right)_{L^2(\Omega)}-(a(x)\,y_n,\varphi)_{L^2(\Omega)}\}\\
	 	 	&\leq&\sup_{\norm{\varphi}{\mathcal{X}}=1}
	 	 	\left\{\norm{\nabla\,y_n}{L^2(\Omega)}\,\norm{\nabla\,\varphi}{L^2(\Omega)}+\norm{\mathcal{B}_{n}\,y_n}{L^{(p+2)'}(\Omega)}\,
	 	 	\norm{\varphi}{L^{p+2}(\Omega)}\right.\\
	 	 	&+&\left. \norm{a}{\infty}\,\norm{y_n}{L^{2}(\Omega)}\,
	 	 	\norm{\varphi}{L^{2}(\Omega)}\right\}<+\infty,\end{eqnarray*}
	 	 so that
	 	 \bq\label{bounded_ynt}\{\partial_t\,y_n\}
	 	 \quad\hbox{ is bounded in }\quad
	 	 L^{\infty}(0,T; \mathcal{X}')\,.\eq
	 	
	 	 Combining \eqref{bounded_yn}, \eqref{bounded_J_{p,n}}, \eqref{bounded_B_{p,n}} and \eqref{bounded_ynt}, it follows that $\{y_n\}$ has a subsequence
	 	 (still denoted by $\{y_n\}$) such that
	 	 \bq\label{weak_yn}y_n&\stackrel{\ast}{\rightharpoonup}&\,y\quad\quad
	 	 \hbox{ in }\quad L^{\infty}(0,T; H_0^1(\Omega))\,.\\
	 	 \label{weak_JnynY}J_{n}(y_n)&\stackrel{\ast}{\rightharpoonup}&\,\mathcal{Y}\quad\,\,\,\,\,
	 	 \hbox{ in }\quad  L^\infty(0,T; L^{p+2}(\Omega))\,.\\
	 	 \label{weak_Bnyn}\mathcal{B}_{n}(y_n)&\stackrel{\ast}{\rightharpoonup}&\,\mathcal{Z}\quad\,\,\,\,\,
	 	 \hbox{ in }\quad L^{\infty}(0,T; L^{(p+2)'}(\Omega))\,.\\
	 	 \label{weak_potyn}\partial_t\,y_n&\rightharpoonup&\,\partial_t\,y\hspace{0.41cm}
	 	 \hbox{ in }\quad L^{(p+2)^\prime}(0,T; \mathcal{X}^\prime)\,.
	 	 \eq
	 	
	 	 \medskip
	 	
	 	 \smallskip
	 	 By Aubin-Lions' Theorem, J. L. Lions, \cite[Lemma 5.2 on page 57]{Lions},  there exist a
	 	 $y\in\,L^2(0,T;L^2(\Om))$
	 	 and a subsequence $\{y_n\}$ (still denoted by $\{y_n\}$) such
	 	 that
	 	 \begin{eqnarray}
	 	 \label{strong_yn}y_n&\rightarrow&\,y\hspace{1.15cm}
	 	 \hbox{ in }\quad L^2(0,T; L^2(\Om))\\
	 	 \label{ae_yn}y_n&\rightarrow&\,y\quad
	 	 \hbox{ a. e.  in }\quad \Omega\,\times\,(0,T)\,.
	 	 \end{eqnarray}
	 	
	 	 \medskip

	 	 Note that the  operator $\mathcal{B}$ is also m-accretive when considered on $ \mathbb{C} $. So,  by Showalter,  \cite[page 211]{Show}, we have that the resolvents $ {J}_n $ given in \eqref{Jrn} are contractions in $ \mathbb{C},$ that is,
	 	 \begin{equation}
	 	 \label{Jrn_contraction}|J_n(z)-J_n(w)| \leq |z-w|, \,\forall\,z,w\in\,\mathbb{C},
	 	 \end{equation}  Note that in the pointwise sense $ \mathcal{B}_n $ and $ J_n $ are essentially the same operators given in the beginning of this section, except that we are considering them on $ \mathbb{C} $ instead $ L^2(\Om). $ \\
	 	
	 	 From above,  let's define
	 	 $$|||\,C\,|||=\inf\{|x|: \,x\in\,C\}.$$ Thanks to Showalter \cite[Proposition 7.1, item c,  page 211]{Show}, we obtain  \begin{equation}\label{prop_show}
	 	 |\mathcal{B}_n(w)| \leq ||| \mathcal{B}(w)|||=|\mathcal{B}(w)|,\,\forall\,w\in\,\mathbb{C},
	 	 \end{equation} where the  equality on the right hand side of \eqref{prop_show} is  due to the fact that the operator $ \mathcal{B} $ given in \eqref{B} is single-valued in $ \mathbb{C}. $\\
	 	
	 	 On the other hand, from \eqref{Bn}, we have
	 	 $\displaystyle w-J_n(w)=\frac1n\,\mathcal{B}_n(w). $
	 	 Thus, combining this fact with
	 	 \eqref{Jrn_contraction} and  \eqref{prop_show}, we obtain
	 	 \begin{equation}\label{ineq}
	 	 \begin{aligned}
	 	 |J_{n}(z)-w|&\leq |J_{n}(z)-J_{n}(w)|+|J_{n}(w)-w|\\
	 	 &\leq |z-w|+\frac{1}{n}\,|\mathcal{B}_{n}(w)|\\
	 	 &\leq|z-w|+\frac{1}{n}\,|\mathcal{B}(w)|,\,\forall\, w, z\in\,\mathbb{C}\,.
	 	 \end{aligned}
	 	 \end{equation}
	 	 It follows from \eqref{ae_yn} that
	 	 \begin{equation}\label{ae_absolute_value}
	 	 |y_n-y|  \rightarrow 0 \quad \text{ a.e. in } \Om \times (0,T)\,.
	 	 \end{equation}
	 	
	 	 \medskip
	 	
	 	 Now, let $ (x,t)\in\,\Omega \times (0,T) $ such that the convergence \eqref{ae_absolute_value} holds and
	 	 $z=y_n$ and $ w=y $ in \eqref{ineq} and letting $ n\rightarrow \infty $,  taking into account \eqref{ae_absolute_value},  it follows that
	 	 \begin{eqnarray}
	 	 \label{ae_Jryn}J_{n}(y_n)&\rightarrow&\,y\quad\hspace{0.55cm}
	 	 \hbox{ a. e.  in }\quad \Om\,\times\,(0,\infty)\,.
	 	 \end{eqnarray}
	 	
	 	 Moreover,  taking into account  (\ref{ae_Jryn}) and the fact that the map  $\mathcal{B}(z)=|z|^p\,z$ is continuous, we infer
	 	 \begin{eqnarray*}
	 	 	\mathcal{B}(J_n(y_n))&\rightarrow&\,\mathcal{B}(y)=|y|^p\,y\quad
	 	 	\hbox{ a. e.  in }\quad \Om\,\times\,(0,\infty)\,.
	 	 \end{eqnarray*}
	 	 Making use of the definition of the Yosida aproximations $ \mathcal{B}_n $ given in  (\ref{BnJn}), it results that
	 	 \begin{eqnarray}
	 	 \label{ae_Bpnyn}\mathcal{B}_{n}(y_n)&\rightarrow&\,|y|^p\,y\quad
	 	 \hbox{ a. e.  in }\quad \Om\,\times\,(0,\infty)\,.  \hspace{.5cm}
	 	 \end{eqnarray}

	 	 Now, combining \eqref{bounded_J_{p,n}}, (\ref{ae_Jryn})   and (\ref{bounded_B_{p,n}}), (\ref{ae_Bpnyn}),   we have, thanks to
	 	 Lions' Lemma, [J. L. Lions, \cite{Lions}, Lemma 1.3, page 12], the following convergences:
	 	 \bq
	 	 \label{weak_JnynRn}J_{n}(y_n)&\rightharpoonup&\,y\hspace{1.01cm}\hbox{ in }
	 	 L^{\infty}(0,T;L^{p+2}(\Om))\,.\\\mathcal{B}_{n}(y_n)&\rightharpoonup&\,|y|^p\,
	 	 y\quad\hbox { in }\label{weak_BpnynRn}
	 	 L^{(p+2)'}(0,T;L^{(p+2)'}(\Om))\,.
	 	 \eq
	 	 So, by convergences \eqref{weak_JnynY}, \eqref{weak_Bnyn} \eqref{weak_JnynRn} and \eqref{weak_BpnynRn},  we get that $\mathcal{Y}=y$ and $ \mathcal{Z}=|y|^p\,y $ almost everywhere in $ \Om\,\times\,(0,T). $
	 	
	 	 Moreover, the convergence \eqref{weak_JnynRn} allows us to infer jointly with  (\ref{weak_yn}) that
	 	 \begin{eqnarray}
	 	 \label{regular1}y\in\,L^{\infty}(0,T;\mathcal{X}).
	 	 \end{eqnarray}
	 	 Finally, let
	 	 $\varphi\in\,C_0^{\infty}([0,T);\mathcal{X})$. Then, from
	 	 (\ref{wn}), we have
	 	 \bqs&&\int_0^T -( y_n(t),\partial_t\,\varphi(t))_{L^2(\Omega)}
	 	 +i\,(\nabla\,y_n(t),\nabla
	 	 \varphi(t))_{L^2(\Omega)}\,dt\\
	 	 \nonumber&&+i\int_0^T[\langle|\,y_n(t)\,|^{p}\,y_n(t),\varphi(t)\rangle_{L^{
	 	 		(p+2)'}(\Omega),\, L^{p+2}(\Omega)}-i(a(x)\,y_n(t),\varphi(t))_{L^{2
	 	 	}(\Omega)}]
	 	 \,dt=0.\eqs
	 	 From  (\ref{weak_yn}) and (\ref{weak_BpnynRn})  by passing to the limit as $n\rightarrow\,\infty,$
	 	 we obtain the variational formula given in (\ref{weak1}).  \\
	 	
	 	 From (\ref{weak1}), it follows that $ y $ belongs to the space
	 	 $$
	 	 \mathcal{W} =\{y\in\,L^{2}(0,T;\mathcal{X})\,\hbox{ such that }\,\partial_t\,y\in\,L^{2}(0,T;\mathcal{X}')\}.
	 	 $$
	 	 Then, employing Showalter [\cite{Show}, proposition 1.2, page 106], we have that $ \mathcal{W} $ can be continuously embedded in the space $ C([0,T]; L^{2}(\Om)) $ and, therefore, combining this fact with (\ref{regular1}), we obtain that
	 	 $y$  satisfies Definition \ref{def1}.	 Moreover, from \eqref{exp}, \eqref{weak_yn} and
	 	 weak lower - semicontinuity of the norm, we obtain the decay estimate \eqref{exp}. Hence, the proof of Theorem \ref{theorem 2.1} is complete.

\section{Unbounded Domains}\label{section3}
	The results presented in this article for bounded domains extend easily to the whole space and exterior domains.  To this end, we consider the damping term $i a(x) y$ with $a(x) \geq a_0 >0$ in $\mathbb{R}^N \backslash B_{R^\prime}$ where $B_{R^\prime}$ represents a ball of radius $R^\prime>0$.
\begin{enumerate}[(i)]
	\item If {$\Omega=\mathbb{R}^N  $,} then we can take some $r>0$ such that $r>R^\prime$ and work locally in the bounded set $B_r$. Following the steps in the proof of Lemma \ref{lema1}, we can find $u=0$ in $B_r\backslash B_{R^\prime}$ and then employ Lemma \ref{Theo. 2.2} to conclude that $u=0$ in $B_{R^\prime}$ as well, and consequently $u=0$ everywhere because the ball $B_r$ was taken arbitrary.	
\item 	Similarly, the result remains valid for an exterior domain $\Omega:=\mathbb{R}^N \backslash \mathcal{O}$, where $\mathcal{O}$ is a compact star-shaped obstacle whose boundary $\Gamma_0$ is smooth and associated with Dirichlet b.c. as in \cite{las03} and  $m(x)\cdot \nu(x) \leq 0  $ on $ \Gamma_0 $. As in the case of the whole space, we can consider a ball $B_{R^\prime}$ which contains the obstacle strictly, namely, $\mathcal{O} \subset \subset B_{R^\prime}$ and we take, as before, $a(x) \geq a_0 >0$ in $\Omega \backslash B_{R^\prime}$. Now, the  observer $ x_0 $ must be taken in the interior of the obstacle $\mathcal{O}$. So, let us consider $r>0$ such that $r>R^\prime$. The idea is to employ Lemma \ref{Theo. 2.2} in order to conclude that if $u=0$ in $(\Omega \cap B_r) \backslash (\Omega \cap B_{R^\prime})$ then $u=0$ in $\Omega \cap B_{R^\prime}$.\vglue.1in
  For an observer $x_0$ located in the interior of the obstacle $\mathcal{O}$,  we have that the inner product $(x-x_0) \cdot \nu(x) \leq 0$ on $\Gamma_0$, namely, $\Gamma_0$ is the uncontrolled or unobserved part according to terminology used in \cite{las03}, so that the unique continuation
	principle presented by \cite{las03} is verified.	
	Finally, it is worth mentioning that in the context of unbounded domains, the convergence \eqref{weak_JnynRn} remains valid by considering  ideas similar to those used in Cavalcanti et al. \cite[(3.43)]{Cavalcanti0}.
	\begin{remark}
		It is important to mention that the UCP developed in  \cite{las03} can be naturally extended to a finite number of the observers $ x_1, x_2, \ldots, x_n $ with a finite number of respective compact star-shaped obstacles $ \mathcal{O}_i $ whose closures are pairwise disjoint.  To this end, one can simply use the following vector field:
		\begin{equation}
		\label{field}
		q(x):=\begin{cases}
x-x_j,\,j=1,\ldots, n \text{ with } x_j\in\,\mathcal{O}_j, x\in\,\Omega\\
\text{and smootly extended in }\Omega\backslash(\mathcal{O}_1 \cup \mathcal{O}_2\cup \ldots \cup \mathcal{O}_n).
		\end{cases}
		\end{equation}
	\end{remark}
\end{enumerate}


\section{Numerical Approximation}
In this final section, we will show some numerical results supporting \ref{theorem 2.1} in $\mathbb{R}^2$. In particular, a Finite Volume scheme is implemented.

\subsection{Presentation of the Scheme.}

We consider that the domain $\Omega \subset \mathbb{R}^2$ in \eqref{problema}. We approximate the domain using an admisible mesh (see \cite{gallouet}) composed by a set $\mathcal{T}$ of convex polygons, denoted as the \textit{control volumes} or \textit{cells}, a set of faces $\mathcal{E}$ contained in hyperplanes of $\mathbb{R}^2$, and a set of points $\mathcal{P}$, representing the centroids of the control volumes. The size of the mesh will be given by $h:= \max_{K\in\mathcal{T}} \{diam(K)\}$. \\

To generate the mesh, we have made use of the open-source code \texttt{PolyMesher} \cite{talischi}, which contructs Vorono\"i tessellations iteratively refined through a Lloyd's method in order to guarantee its regularity.

We will denote by $K \in \mathcal{T}$ a control volume or cell inside the mesh, which in turn has centroid $x_K \in \mathbb{R}^2$, a measure $m(K)$ (in our case: the area of $K$), a set of neighboring cells $\mathcal{N}(K)$, and a set $\mathcal{E}_K$ of faces $\sigma \in \mathcal{E}_K \subset \mathcal{E} = \mathcal{E}_{int} \cup \mathcal{E}_{ext}$, where $\mathcal{E}_{int}$ is the set of inner faces and $\mathcal{E}_{ext}$ is the set of boundary faces. We will also write $t_n = n\Delta t$ for a given timestep $\Delta t$. We will denote $y_K^n$ as the numerical approximation of the solution of problem \eqref{problema} over the cell $K$ at the time $t_n$. We will also write $y_K^{n+\frac{1}{2}}:= \frac{y_K^{n+1} + y_K^n}{2}$. $\forall K \in \mathcal{T}$, the proposed Finite Volume scheme for this problem will be defined as follows;
\small
\begin{equation}\label{esq_num}
\begin{cases}  i m(K_i)\frac{y_K^{n+1} - y_K^n}{\Delta t} + \sum_{\sigma \in \mathcal{E}_K} F^{n+\frac{1}{2}}_{K,\sigma} - \frac{m(K)}{2p}\frac{|y_K^{n+1}|^{2p} - |y_K^{n}|^{2p}}{|y_K^{n+1}|^2 - |y_K^n|^2}(y_K^{n+1} + y_K^n) + im(K)a(x_K)y_K^{n+\frac{1}{2}} = 0 \\
  F^n_{K,\sigma} = \tau_{\sigma} (y_L^n - y_K^n), \quad \sigma \in \mathcal{E}_{int}, \: \sigma = K|L,\quad L \in \mathcal{T} \\
  F^n_{K,\sigma} = -\tau_\sigma y_K^n,\quad \sigma \in \mathcal{E}_{ext}: \sigma \in \mathcal{E}_K \\
  \tau_\sigma = m(\sigma)/|x_K - x_L|,\quad \sigma \in \mathcal{E}_{int}, \quad L \in \mathcal{T}: \sigma = K|L \\
  \tau_\sigma = m(\sigma)/d(x_K,\sigma), \quad \sigma \in \mathcal{E}_{ext}: \sigma \in \mathcal{E}_K
\end{cases}
\end{equation}
\normalsize
The discretization of the nonlinear term comes from the work of Delfour, Fortin and Payre \cite{delfour}, which was proposed in order to preserve the Energy at $H^1$ level if there is no damping term. The numerical solution over he whole domain $[0,T]\times \Omega$ will de denoted by $y_{\mathcal{T},\Delta t}$, such that $y_{\mathcal{T},\Delta t}(x_K,t_n) = y_K^n$. In some cases, we will write $y^n$ instead of $y_{\mathcal{T},\Delta t}(t_n)$ for the sake of clarity.

Given the symmetric structure of the matrix involved in the induced linear system of equations, a GMRES method is used to solve it. The nonlinear problem is solved using a Picard Fixed Point iteration with a tolerance equal to $10^{-6}$ before moving to the next timestep.

\subsection{Properties and convergence analysis.}

In order to state the properties of the scheme \eqref{esq_num}, we will need some notation. We will denote the discrete $L^2$ norm as follows:
\begin{equation*}
  ||y^n||^2_{L^2_\mathcal{T}(\Omega)} := \sum_{K\in\mathcal{T}} |y_K^n|^2 m(K).
\end{equation*}
In a similar fashion, we define the discrete $L^{2p}$ norm as
\begin{equation*}
      ||y^n||^{2p}_{L^{2p}_{\mathcal{T}}(\Omega)} := \sum_{K\in \mathcal{T}} |y_K^n|^{2p} m(K).
\end{equation*}
The discrete version of the $H_0$ norm will be defined as:
\begin{equation*}
  ||y^n||^2_{H_{0,\mathcal{T}}^1(\Omega)} = \sum_{\sigma \in \mathcal{E}}\tau_\sigma |D_\sigma y^n|^2,
\end{equation*}
where $\tau_\sigma$ is defined as in \eqref{esq_num}, and for $K\in\mathcal{T}$ and $L\in \mathcal{N}(K)$,
\begin{equation*}
  D_\sigma y^n =
  \begin{cases}
    y_L^n - y_K^n,\quad \text{ if } \sigma=K|L \in \mathcal{E}_{int} \\
    -y_K^n,\qquad \quad \: \text{   if } \sigma \in \mathcal{E}_{ext}.
  \end{cases}
\end{equation*}
The following property holds:
\begin{theorem}\label{converg_num}
  The numerical scheme \eqref{esq_num} admits the existence of a unique solution $y_{\mathcal{T},\Delta t}$.
\end{theorem}
\begin{proof}
  For a given $n \in \{0,1,\dots,N\}$, and assuming that $y_K^n = 0, \forall K \in \mathcal{T}$, we take \eqref{esq_num} and multiply it by $\overline{y}_K^{n+1}$, sum over $K\in \mathcal{T}$, and extract the imaginary part. This will lead us to conclude that $y_K^{n+1} = 0,\: \forall K \in \mathcal{T}$, and hence the existence of solutions is proved. Uniqueness follows after noticing that the linear system induced by the numerical scheme has finite dimension with respect to the vector of unknowns $y_K^{n+1}$, and hence has unique solution.
\end{proof}
Let us define the discrete version of the mass functional $E_0(y(t))$ as follows:
  \begin{equation*}
    E_0^{(n)} := \frac{1}{2}\sum_{K\in \mathcal{T}}|y_K^{n}|^2m(K), \quad n \in \mathbb{N}.
  \end{equation*}
If we multiply the numerical scheme by $\overline{y}_K^{n+\frac{1}{2}}$, sum over $K\in \mathcal{T}$, and extract the imaginary part, we get the following result:
\begin{theorem}
  If $a(x) \equiv 0,\: \forall x\in \Omega$ in \eqref{esq_num}, then the following property is true $\forall n\in \mathbb{N}$:
  \begin{equation}\label{l2_num}
    E_0^{(n)} = E_0^{(n+1)}
  \end{equation}
  If $a(x) \geq a_0 > 0,\: x\in \omega \subset \Omega$, then
  \begin{equation*}
    E_0^{(0)} \geq E_0^{(n)}, \quad \forall n \in \mathbb{N}.
  \end{equation*}
\end{theorem}
A consequence of the previous procedure reads as follows:
\begin{corollary}
  Let $y_{\mathcal{T},\Delta t}$ be the solution of \eqref{esq_num} such that $ E_0^{(0)} < \infty$. Then, there exists a constant $C_\infty$, depending on $y^0$ and $T$, such that
  \begin{equation}\label{lInf}
    ||y_{\mathcal{T},\Delta t}||_\infty < C_\infty
  \end{equation}
  where $||y^n||_\infty := \max_{K\in \mathcal{T}}|y^n_K|$.
\end{corollary}

We will also define the discrete version of the Energy functional at $H^1$ level:
\begin{equation}\label{eH1}
  E_1^{(n)} := \frac{1}{2}\sum_{\sigma \in \mathcal{E}}\tau_\sigma |D_\sigma y^n|^2 + \sum_{K\in\mathcal{T}} \frac{1}{2p}|y_K^n|^{2p} m(K)
\end{equation}
The following property holds:
\begin{theorem}\label{E_h1}
  Let $y_{\mathcal{T},\Delta t}$ be the numerical solution induced by the scheme \eqref{esq_num} such that $||y_{\mathcal{T},\Delta t}^0||_{L^2_{\mathcal{T}(\Omega)}}^2 < \infty$. If $a(x) \equiv 0,\: \forall x \in \Omega$ in \eqref{esq_num}; then the following property holds true $\forall n \in \mathbb{N}$:
  \begin{equation}\label{ener_cons}
    E_1^{(n+1)} = E_1^{(n)}.
  \end{equation}
  If $a(x) \geq a_0 > 0,\; x \in \omega \subset \Omega$ and $a(x) \in W^{1,\infty}(\Omega)$, then there exists a constant $C>0$, depending on $T$, $a(x)$, and $y^0$, such that
    \begin{equation}\label{ener_acotada}
    E_1^{(n)} \leq E_1^{(0)} + C.
  \end{equation}
\end{theorem}
\begin{proof}
  We multiply \eqref{esq_num} by $\frac{\overline{y}_K^{n+1} - \overline{y}_K^n}{\Delta t}$, sum over $K\in\mathcal{T}$, and extract the real part. We get
  \begin{align}
    Re\Bigg(\sum_{K\in\mathcal{T}}\sum_{\sigma \in \mathcal{E}_K} &F^{n+\frac{1}{2}}_{K,\sigma}\frac{\overline{y}_K^{n+1} - \overline{y}_K^n}{\Delta t}\Bigg) - \sum_{K\in\mathcal{T}} \frac{m(K)}{2p\Delta t}\Big(|y_K^{n+1}|^{2p} - |y_K^{n}|^{2p} \Big) \label{eq3.0} \\
    + & Re\Big(i\sum_{K\in\mathcal{T}} m(K)a(x_K)y_K^{n+\frac{1}{2}}\frac{\overline{y}_K^{n+1}-\overline{y}_K^n}{\Delta t}\Big)  = 0 \nonumber.
  \end{align}
  After using the identity $Re(a(\overline{b}-\overline{a})) = \frac{1}{2}\big(|b|^2-|a|^2 - |b-a|^2\big)$ for $a,b\in \mathbb{C}$, and reordening the sum, the first term in \eqref{eq3.0} becomes
  \begin{align*}
    Re\Bigg(\sum_{K\in\mathcal{T}}\sum_{\sigma \in \mathcal{E}_K} F^{n+\frac{1}{2}}_{K,\sigma}\frac{\overline{y}_K^{n+1} - \overline{y}_K^n}{\Delta t}\Bigg) &= \sum_{\sigma \in \mathcal{E}}\frac{\tau_\sigma}{2} \Big(|y_L^n - y_K^n|^2 - |y_L^{n+1} - y_K^{n+1}|^2 \Big) \\
    &= \sum_{\sigma \in \mathcal{E}}\frac{\tau_\sigma}{2}\Big(|D_\sigma y^n|^2 - |D_\sigma^{n+1}|^2\Big).
  \end{align*}
  With this, \eqref{eq3.0} turns into the following:
  \begin{equation}
    \frac{1}{\Delta t}E_1^{(n+1)} = \frac{1}{\Delta t}E_1^{(n)} +  Re\Big(i\sum_{K\in\mathcal{T}} m(K)a(x_K)y_K^{n+\frac{1}{2}}\frac{\overline{y}_K^{n+1}-\overline{y}_K^n}{\Delta t}\Big). \label{eq3.0.1}
  \end{equation}
  If $a(x) \equiv 0$, then we get \eqref{ener_cons}. If not, then we will need to recall the following from the numerical scheme:
  \begin{equation}
    \frac{y_K^{n+1} - y_K^n}{\Delta t} = \frac{i}{m(K)}\sum_{\sigma \in \mathcal{E}_K} F^{n+\frac{1}{2}}_{K,\sigma} - \frac{i}{2p}\frac{|y_K^{n+1}|^{2p} - |y_K^{n}|^{2p}}{|y_K^{n+1}|^2 - |y_K^n|^2}(y_K^{n+1} + y_K^n) - a(x_K)y_K^{n+\frac{1}{2}}.  \label{eq3.0.2}
  \end{equation}
  Replacing \eqref{eq3.0.2} in \eqref{eq3.0.1} will lead us to study the following:
  \small
  \begin{align}
    i\sum_{K\in\mathcal{T}} m(K)a(x_K)y_K^{n+\frac{1}{2}}\frac{\overline{y}_K^{n+1}-\overline{y}_K^n}{\Delta t} &= \sum_{K\in\mathcal{T}} a(x_K)y_K^{n+\frac{1}{2}}\sum_{\sigma \in \mathcal{E}_K} \overline{F}^{n+\frac{1}{2}}_{K,\sigma} \nonumber \\
    & \: - \sum_{K\in\mathcal{T}} a(x_K) \frac{m(K)}{p}\frac{|y_K^{n+1}|^{2p} - |y_K^{n}|^{2p}}{|y_K^{n+1}|^2 - |y_K^n|^2}\big|y_K^{n+\frac{1}{2}}\big|^2 \label{eq3.0.3} \\
    & \: - i\sum_{K\in\mathcal{T}} m(K)(a(x_K))^2|y_K^{n+\frac{1}{2}}|^2. \nonumber
  \end{align}
 \normalsize
 After extracting the real part in \eqref{eq3.0.3} the third term at the right hand side vanishes and the second term is a strictly negative number. For the first term, again using the identity $Re(a(\overline{b}-\overline{a})) = \frac{1}{2}\big(|b|^2-|a|^2 - |b-a|^2\big)$ and reordering the sum, we get
 \footnotesize
 \begin{align}
   Re\Bigg(\sum_{K\in\mathcal{T}} a(x_K)y_K^{n+\frac{1}{2}}\sum_{\sigma \in \mathcal{E}_K} \overline{F}^{n+\frac{1}{2}}_{K,\sigma}\Bigg) &= \sum_{K\in\mathcal{T}}\sum_{\sigma \in \mathcal{E}_K} \frac{\tau_\sigma}{8}\Big(|y_K^{n+1}|^2 + |y_K^{n}|^2\Big) \big(a(x_L)-a(x_K) \big) \nonumber \\
   & -  \sum_{K\in\mathcal{T}}\sum_{\sigma \in \mathcal{E}_K} \frac{\tau_\sigma}{8}a(x_K)\Big(|y_L^{n+1} - y_K^{n+1}|^2 + |y_L^{n} - y_K^{n}|^2\Big) \label{eq3.0.4} \\
   & +  \sum_{K\in\mathcal{T}}\sum_{\sigma \in \mathcal{E}_K} \frac{\tau_\sigma}{4}a(x_K) Re\Big(y_K^{n+1}(\overline{y}_L^n - \overline{y}_K^n) + y_K^n(\overline{y}_L^{n+1} - \overline{y}_K^{n+1}) \Big). \nonumber
 \end{align}
\normalsize
The second term in \eqref{eq3.0.4} is strictly negative. Hence, and given the regularity condition of the damping function $a(x) \in W^{1,\infty}(\Omega)$, we can infer the existence of a constant $C_1$, depending on $a(x)$, such that
\small
 \begin{align*}
   Re\Bigg(\sum_{K\in\mathcal{T}} a(x_K)y_K^{n+\frac{1}{2}}\sum_{\sigma \in \mathcal{E}_K} \overline{F}^{n+\frac{1}{2}}_{K,\sigma}\Bigg) &\leq C_1 \sum_{K\in\mathcal{T}}\sum_{\sigma \in \mathcal{E}_K} \frac{\tau_\sigma}{8}\Big(|y_K^{n+1}|^2 + |y_K^{n}|^2\Big)  \nonumber \\
   & +  C_1\sum_{K\in\mathcal{T}}\sum_{\sigma \in \mathcal{E}_K} \frac{\tau_\sigma}{4}\Big|y_K^{n+1}(\overline{y}_L^n - \overline{y}_K^n) + y_K^n(\overline{y}_L^{n+1} - \overline{y}_K^{n+1}) \Big| \nonumber \\
   & \leq \frac{C_1}{8}\Big(||y^{n+1}||_{L^2_{\mathcal{T}(\Omega)}}^2 + ||y^{n}||_{L^2_{\mathcal{T}(\Omega)}}^2 \Big) \\
   & + \frac{C_1}{4}\sum_{K\in\mathcal{T}}\sum_{\sigma \in \mathcal{E}_K} \tau_\sigma \Big( 4|y_K^{n+1}|^2 + 4|y_K^n|^2 \Big) \\
   & \leq \frac{9}{4}C_1 ||y^0||_{L^2_{\mathcal{T}(\Omega)}}^2.
 \end{align*}
 \normalsize
 Hence, \eqref{eq3.0.1} will turn into
\begin{align*}
  \frac{1}{\Delta t}E_1^{(n+1)} &= \frac{1}{\Delta t}E_1^{(n)} +  Re\Big(i\sum_{K\in\mathcal{T}} m(K)a(x_K)y_K^{n+\frac{1}{2}}\frac{\overline{y}_K^{n+1}-\overline{y}_K^n}{\Delta t}\Big) \\
  &\leq \frac{1}{\Delta t}E_1^{(n)} + \frac{9}{4} C_1 ||y^0||_{L^2_{\mathcal{T}(\Omega)}}^2.
\end{align*}
Multiplying the previous result by $\Delta t$ and repeating the upper bound $n$ times will lead us to
\begin{equation*}
  E_1^{(n+1)} \leq E_1^{(0)} + \frac{9C}{4}n\Delta t ||y^0||_{L^2_{\mathcal{T}(\Omega)}}^2,
\end{equation*}
and because $||y^0||_{L^2_{\mathcal{T}(\Omega)}}^2 < \infty$, we can infer the existence of a constant $C$, depending on $T$, $y^0$, and $a(x)$, such that
\begin{equation*}
  E_1^{(n+1)} \leq E_1^{(0)} + C.
\end{equation*}
Thus, the theorem is proved.
\end{proof}
On the other hand, if we go back to \eqref{ener_acotada} and compare it with the definition \eqref{eH1}, we get the following result:
\begin{corollary}
  Let $y^n$ be the solution of \eqref{esq_num} such that $||y^0||_{L^2_{\mathcal{T}}(\Omega)}^2 < \infty$ and $E_1^{(0)} < \infty$. Then, there exist some constants $C_1$ and $C_2$, depending on $y^0$, $a(x)$, and $T$, such that
  \begin{equation}\label{cotaH1}
    ||y^n||_{H_{0,\mathcal{T}}^1(\Omega)} < C_1,\quad \forall n\in\mathbb{N}.
  \end{equation}
 and
  \begin{equation}\label{cotaL2p}
    ||y^n||_{L^{2p}_{\mathcal{T}}(\Omega)} < C_2,\quad \forall n\in\mathbb{N}.
  \end{equation}
\end{corollary}

This upper bound will help us to prove the convergence of the numerical scheme.
\begin{theorem}
  For $m \in \mathbb{N}$, let $\{y_m\}_{m\in \mathbb{N}},\: y_m = y_{\mathcal{T}_m,\Delta t_m}(x,t)$ be a sequence of solutions of \eqref{esq_num} induced by their respective initial conditions $\{y_m^{0}\}_{m\in \mathbb{N}} \subset \mathcal{X}$, while using a sequence of admissible meshes $\mathcal{T}_m$ and timesteps $\Delta t_m$ such that $h_m \rightarrow 0$ and $\Delta t_m \rightarrow 0$ when $m\rightarrow \infty$. Then, there exists a subsequence of the sequence of numerical solutions, still denoted by $\{y_m\}_{m\in \mathbb{N}}$, which converges to the weak solution $y(t)$ given by the Definition \ref{def1} when $m \rightarrow \infty$.
\end{theorem}
\begin{proof}
  We will start by proving that $\partial_t y_m$ is bounded in $\mathcal{X}'$; this is
  \small
  \begin{align}
    ||\partial _t y_m||_{\mathcal{X}_m'} &:= \sup_{||\varphi||_{\mathcal{X}_m} =1} \Big\{ \Big| \big(\partial y_m,\varphi \big)_{L^2_{\mathcal{T}_m}(\Omega)} \Big| \Big\} \nonumber \\
    &= \sup_{||\varphi||_{\mathcal{X}_m} =1} \Bigg\{ \Bigg| i\bigg(\sum_{K\in\mathcal{T}_m}\sum_{\sigma \in \mathcal{E}_K} \tau_\sigma \big(y_L^{n+\frac{1}{2}} - y_K^{n+\frac{1}{2}}\big)  \overline{\varphi}_K\bigg)  \nonumber  \\
    & \quad -\frac{i}{2p} \sum_{K\in\mathcal{T}_m}  \bigg(\frac{|y_K^{n+1}|^{2p} - |y_K^{n}|^{2p}}{|y_K^{n+1}|^2 - |y_K^n|^2}(y_K^{n+1} + y_K^n) \overline{\varphi}_K m(K) \bigg) - \sum_{K\in\mathcal{K}} \big(a(x_K)y_K^{n+\frac{1}{2}}\overline{\varphi}_K m(K) \big) \Bigg|  \Bigg\} \label{eq3.0.5} \\
    & \quad < \infty \nonumber.
  \end{align}
  \normalsize
  The first term in the right hand side of \eqref{eq3.0.5} can be rewritten as follows
    \begin{align}
    \sum_{n=0}^N &\sum_{K\in \mathcal{T}} \sum_{\sigma\in \mathcal{E}_K} \tau_{\sigma} (y_L^{n+\frac{1}{2}} - y_K^{n+\frac{1}{2}})\overline{\varphi}_K \Delta t = \nonumber \\
    &\sum_{n=0}^N \sum_{K|L\in \mathcal{E}_{int}} m(K|L) (y_L^{n+\frac{1}{2}} - y_K^{n+\frac{1}{2}})\frac{\overline{\varphi}_K - \overline{\varphi}_L}{d_{K|L}} \Delta t. \nonumber
    \end{align}
    After \eqref{cotaH1} and the regularity of $\varphi$, we can write
    \begin{equation}\label{eq3.0.6}
       \sum_{K\in\mathcal{T}_m}\sum_{\sigma \in \mathcal{E}_K} \tau_\sigma \big(y_L^{n+\frac{1}{2}} - y_K^{n+\frac{1}{2}}\big)  \overline{\varphi}_K < \infty.
    \end{equation}
    For the second term in \eqref{eq3.0.5}, we will consider three cases.
    \begin{itemize}
    \item[$\bullet$] If $p \leq 1$, we have
      \begin{align*}
        \Big|\sum_{K\in\mathcal{T}_m} \frac{|y_K^{n+1}|^{2p} - |y_K^{n}|^{2p}}{|y_K^{n+1}|^2 - |y_K^n|^2}(y_K^{n+1} + y_K^n) \overline{\varphi}_K m(K)\Big| &\leq \sum_{K\in\mathcal{T}_m} |(y_K^{n+1} + y_K^n) \overline{\varphi}_K| m(K)  \nonumber \\
        & \leq 2||\varphi||_{L_{\mathcal{T}_m}^\infty (\Omega)} ||y^0||^2_{L_{\mathcal{T}_m}^2 (\Omega)}
      \end{align*}
      which is bounded. \\

    \item[$\bullet$] If $1 < p < 2$, then
      \scriptsize
      \begin{align}
        \Big|\sum_{K\in\mathcal{T}_m} \frac{|y_K^{n+1}|^{2p} - |y_K^{n}|^{2p}}{|y_K^{n+1}|^2 - |y_K^n|^2}(y_K^{n+1} + y_K^n) \overline{\varphi}_K m(K)\Big| &\leq  2||\varphi||_{L_{\mathcal{T}_m}^\infty}||y^0||_{L_{\mathcal{T}_m}^\infty} \sum_{K\in\mathcal{T}_m} \bigg(|y_k^{n+1}|^{2p-2} + |y_K^n|^{2p-2} \bigg) m(K). \nonumber
      \end{align}
      \normalsize
      Using Young's inequality, we get
            \scriptsize
      \begin{align*}
        \sum_{K\in\mathcal{T}_m} \bigg(|y_k^{n+1}|^{2p-2} + |y_K^n|^{2p-2} \bigg) m(K) &\leq \sum_{K\in\mathcal{T}_m} \bigg(\bigg(\frac{2p-2}{2p}\bigg)\big(|y_k^{n+1}|^{2p} + |y_K^n|^{2p}\big) + \frac{2}{p} \bigg) m(K) \nonumber
      \end{align*}
      \normalsize
      which is also bounded due to \eqref{cotaL2p}, \eqref{lInf}, and by the fact that $|\Omega| < \infty$. \\

    \item[$\bullet$] If $p \geq 2$, then we have
      \scriptsize
      \begin{align}
        \Big|\sum_{K\in\mathcal{T}_m} \frac{|y_K^{n+1}|^{2p} - |y_K^{n}|^{2p}}{|y_K^{n+1}|^2 - |y_K^n|^2}(y_K^{n+1} + y_K^n) \overline{\varphi}_K m(K)\Big| &\leq  2||\varphi||_{L_{\mathcal{T}_m}^\infty}||y^0||_{L_{\mathcal{T}_m}^\infty} \sum_{K\in\mathcal{T}_m} \frac{p}{2}\bigg(|y_k^{n+1}|^{2p-2} + |y_K^n|^{2p-2} \bigg) m(K) \nonumber
      \end{align}
      \normalsize
      which is bounded by the same reasons argued in the previous point.
    \end{itemize}
    Hence, we conclude that the second term in \eqref{eq3.0.5} is bounded for any $p > 0$; this is,
    \begin{align}
      \Big|\sum_{K\in\mathcal{T}_m} \frac{|y_K^{n+1}|^{2p} - |y_K^{n}|^{2p}}{|y_K^{n+1}|^2 - |y_K^n|^2}(y_K^{n+1} + y_K^n) \overline{\varphi}_K m(K)\Big| & < \infty. \label{eq3.0.7}
    \end{align}
    Regarding the third term in \eqref{eq3.0.5}: thanks to \eqref{l2_num}, and the regularity properties of $a(x)$, we observe that
    \begin{equation}\label{eq3.0.8}
      \sum_{K\in\mathcal{K}} a(x_K)y_K^{n+\frac{1}{2}}\overline{\varphi}_K m(K) \leq \frac{C_2}{2} \Big(||y^{n+\frac{1}{2}}||_{L_{\mathcal{T}_m}^2(\Omega)}^2 + ||\varphi||_{L_{\mathcal{T}_m}^2(\Omega)}^2 \Big) < \infty
    \end{equation}
    where $C_2$ is a constant depending on $a(x)$. Combining \eqref{eq3.0.6}, \eqref{eq3.0.7} and \eqref{eq3.0.8}, we conclude that
    \begin{equation}
      \{\partial_t y_m\} \quad \text{ is bounded in } \quad L^\infty (0,T;\mathcal{X}').
    \end{equation}
    Therefore, due to the fact that
    \begin{equation*}
      H_0^1(\Omega)\stackrel{c}{\hookrightarrow}\,L^2(\Omega)\hookrightarrow\,H^{-2}(\Omega),
    \end{equation*}
    and thanks to the Aubin-Lions Theorem, we can extract a subsequence, still denoted by $\{y_m\}_{m\in \mathbb{N}}$, such that
    \begin{equation}\label{convergL2}
      y_m \rightarrow y \quad \text{ strongly in } \quad L^2 (0,T;L^2 (\Omega)).
    \end{equation}
    We will now prove that this $y$ is the weak solution given by Definition \ref{def1}. Let $\varphi \in C_0^\infty (0,T;\mathcal{X})$ such that $\nabla \varphi \cdot \mathbf{\hat{n}} = 0$ in $\partial \Omega \times [0,T]$. Multiplying the numerical scheme \eqref{esq_num} by $\frac{\Delta t}{2}\Big(\overline{\varphi}(x_K,n\Delta t) + \overline{\varphi}(x_K,(n+1)\Delta t)\Big) =: \frac{\Delta t}{2}\overline{\varphi}(x_K,t_{n+\frac{1}{2}})$, and summing over $K\in \mathcal{T}$ and over $n = 0,\:\dots,\: N$ with $T = N\Delta t$, we get:
  \small
  \begin{align}
    &i\sum_{n=0}^N \sum_{K\in \mathcal{T}} m(K) (y_K^{n+1}- y_K^n)\overline{\varphi}(x_K,t_{n+\frac{1}{2}}) + \sum_{n=0}^N \sum_{K\in \mathcal{T}} \sum_{\mathcal{N}(K)}  \tau_{K|L} (y_L^{n+\frac{1}{2}} - y_K^{n+\frac{1}{2}})\overline{\varphi}(x_K,t_{n+\frac{1}{2}}) \Delta t \nonumber \\
    & \quad - \sum_{n=0}^N \sum_{K\in \mathcal{T}} |y_k^{n+\frac{1}{2}}|^py_K^{n+\frac{1}{2}}\overline{\varphi}(x_K,t_{n+\frac{1}{2}})\Delta t + i\sum_{n=0}^N \sum_{K\in \mathcal{T}}  a(x_K)y_K^{n+\frac{1}{2}}\overline{\varphi}(x_K,t_{n+\frac{1}{2}}) \Delta t = 0. \label{eq3.1}
  \end{align}
  \normalsize
  We can re-write the first term in \eqref{eq3.1}, after using summation by parts and recalling that $\varphi \in C_0^\infty (0,T;\mathcal{X})$:
  \small
  \begin{equation}
    i\sum_{n=0}^N \sum_{K\in \mathcal{T}} m(K) (y_K^{n+1}- y_K^n)\overline{\varphi}(x_K,t_{n+\frac{1}{2}}) = -i\sum_{n=0}^N \sum_{K\in \mathcal{T}} m(K) y_K^{n}\Big(\frac{\overline{\varphi}(x_K,t_{n+1})- \overline{\varphi}(x_K,t_{n-1})}{2} \Big). \nonumber
  \end{equation}
  \normalsize
  Hence, because $\{y_m\}_{m\in\mathcal{N}}$ is bounded in $L^\infty((0,T)\times,L^2( \Omega))$, then as $m\rightarrow \infty$,
  \begin{equation}
    -i\sum_{n=0}^N \sum_{K\in \mathcal{T}} m(K) y_K^{n}\Big(\frac{\overline{\varphi}(x_K,t_{n+1})- \overline{\varphi}(x_K,t_{n-1})}{2} \Big) \rightarrow -i\int_0^T \int_\Omega y(x,t) \overline{\varphi}_t(x,t) dxdt. \label{eq3.2}
  \end{equation}
  The second term in \eqref{eq3.1} can also be re-written as follows:
  \begin{align}
    \sum_{n=0}^N &\sum_{K\in \mathcal{T}} \sum_{L\in \mathcal{N}(K)} \tau_{K|L} (y_L^{n+\frac{1}{2}} - y_K^{n+\frac{1}{2}})\overline{\varphi}(x_K,t_{n+\frac{1}{2}}) \Delta t = \nonumber \\
    &\sum_{n=0}^N \sum_{K|L\in \mathcal{E}_{int}} m(K|L) (y_L^{n+\frac{1}{2}} - y_K^{n+\frac{1}{2}})\frac{\overline{\varphi}(x_K,t_{n+\frac{1}{2}}) - \overline{\varphi}(x_L,t_{n+\frac{1}{2}})}{d_{K|L}} \Delta t. \label{eq3.3}
  \end{align}
  On the other hand,
  \scriptsize
  \begin{align}
    \sum_{n=0}^N \int_{n\Delta t}^{(n+1)\Delta t} \int_\Omega y_{\mathcal{T},\Delta t}(x,t) \Delta \overline{\varphi}(x,n\Delta t) dx dt &= \sum_{n=0}^n \sum_{K\in \mathcal{T}} y_K^{n+\frac{1}{2}}\int_K \Delta \overline{\varphi}(x,t_{n+\frac{1}{2}})dx \Delta t \label{eq3.4.0}  \\
    &= \sum_{n=0}^N \sum_{K|L\in \mathcal{E}_{int}}\Big(y_K^{n+\frac{1}{2}} - y_L^{n+\frac{1}{2}} \Big)\int_{K|L}\nabla \overline{\varphi}(x,t_{n+\frac{1}{2}})\cdot \mathbf{n}_{K,L}d\gamma. \label{eq3.4}
  \end{align}
  \normalsize
  By the same reasons argued in \eqref{eq3.2}, we have that
  \begin{equation}
    \sum_{n=0}^N \int_{n\Delta t}^{(n+1)\Delta t} \int_\Omega y_{\mathcal{T},\Delta t}(x,t) \Delta \overline{\varphi}(x,n\Delta t) dx dt \rightarrow \int_0^T \int_\Omega y(x,t) \Delta \overline{\varphi}(x,t) dx dt \label{eq3.5}
  \end{equation}
  as $m\rightarrow \infty$. Now, subtracting the right hand side of \eqref{eq3.3} from \eqref{eq3.4},
  \scriptsize
  \begin{align}
    \sum_{n=0}^N \sum_{K|L\in \mathcal{E}_{int}}m(K|L)\Big(y_K^{n+\frac{1}{2}} - y_L^{n+\frac{1}{2}} \Big)\Bigg(\int_{K|L}\nabla \overline{\varphi}(x,t_{n+\frac{1}{2}})\cdot \mathbf{n}_{K,L}d\gamma - \frac{\overline{\varphi}(x_K,t_{n+\frac{1}{2}}) - \overline{\varphi}(x_L,t_{n+\frac{1}{2}})}{d_{K|L}}\Bigg) \Delta t. \label{eq3.6}
  \end{align}
  \normalsize
  Because of the regularity properties of $\varphi$, we have that \eqref{eq3.6} goes to $0$ when $m\rightarrow \infty$. Hence, and thanks to \eqref{eq3.4.0} and \eqref{eq3.5},
  \small
  \begin{equation*}
    \sum_{n=0}^N \sum_{K|L\in \mathcal{E}_{int}} m(K|L) (y_L^{n+\frac{1}{2}} - y_K^{n+\frac{1}{2}})\frac{\overline{\varphi}(x_K,t_{n+\frac{1}{2}}) - \overline{\varphi}(x_L,t_{n+\frac{1}{2}})}{d_{K|L}} \Delta t \rightarrow \int_0^T \int_\Omega y(x,t) \Delta \overline{\varphi}(x,t) dx dt.
  \end{equation*}
  \normalsize
  The third and fourth terms in \eqref{eq3.1} can be treated in a similar way because $y_m \in L^\infty (0,T;\mathcal{X})$; hence, and due to \eqref{convergL2}, we have
  \small
  \begin{equation*}
    \sum_{n=0}^N \sum_{K\in \mathcal{T}} |y_K^{n+\frac{1}{2}}|^py_K^{n+\frac{1}{2}}\overline{\varphi}(x_K,t_{n+\frac{1}{2}})\Delta t \rightarrow \int_0^T \int_\Omega  |y(x,t)|^py(x,t)\overline{\varphi}(x,t)dxdt,\: \text{ as } m\rightarrow \infty.
  \end{equation*}
  \normalsize
  Finally,
  \small
  \begin{equation*}
    i\sum_{n=0}^N \sum_{K\in \mathcal{T}}  a(x_K)y_K^{n+\frac{1}{2}}\overline{\varphi}(x_K,t_{n+\frac{1}{2}}) \Delta t \rightarrow i\int_0^T \int_\Omega a(x)y(x,t)\overline{\varphi}(x,t)dxdt,\: \text{ as }m\rightarrow \infty.
  \end{equation*}
  Thus, when passing to the limit in \eqref{eq3.1} and integrating by parts, we conclude that $y$ is the weak solution of \eqref{problema}; concluding the proof.
\end{proof}

\subsection{Example I}
\begin{figure}[htb!]
  \centering
  \begin{subfigure}[t]{0.45\textwidth}
    \includegraphics[width=\textwidth]{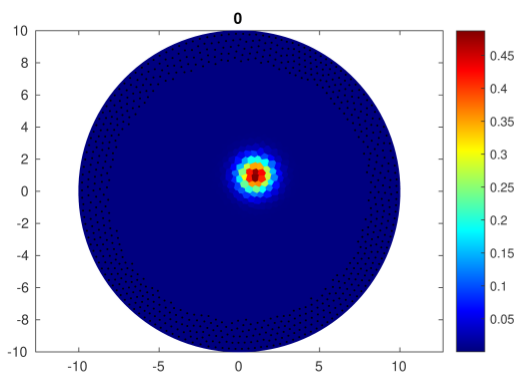}
  \end{subfigure}
  \begin{subfigure}[t]{0.45\textwidth}
    \includegraphics[width=\textwidth]{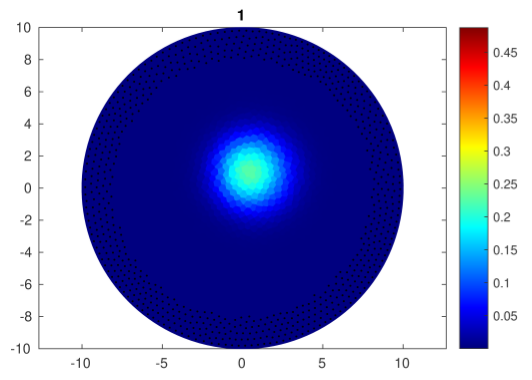}
  \end{subfigure} \\
    \begin{subfigure}[t]{0.45\textwidth}
      \includegraphics[width=\textwidth]{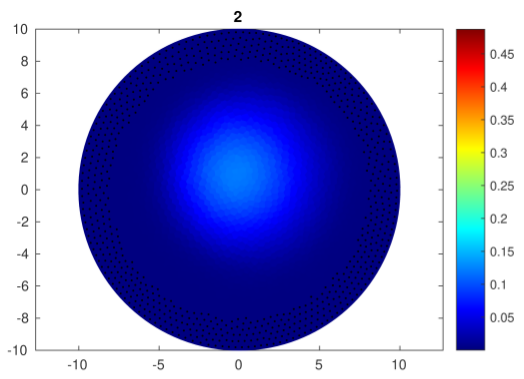}
    \end{subfigure}
  \begin{subfigure}[t]{0.45\textwidth}
    \includegraphics[width=\textwidth]{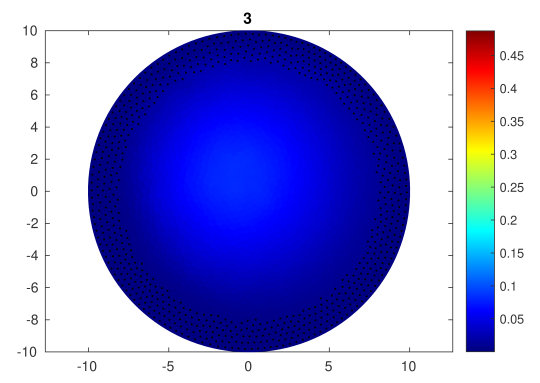}
  \end{subfigure}
    \caption{Numerical solution at different timesteps. Cells with black dots indicate the zone where the damping function is in place.}
    \label{fig1}
\end{figure}
In the following example, we will use the given numerical scheme to solve equation \eqref{problema} for $p = 2$, $T = 500$, $\Omega$ being disk with ratio $r = 10$, $\omega \subset \Omega: x^2+y^2 > 8^2$, and a damping function defined as follows:
\begin{equation*}
  a(x,y) =
  \begin{cases}
    \big(\sqrt{(x^2 + y^2)} - 8\big)^2, \quad 8^2 \leq x^2 + y^2 \leq 10^2 \\
    0, \qquad \qquad \qquad \text{ otherwise.}
  \end{cases}
\end{equation*}
Observe that the damping fulfills condition \eqref{anabla}. The initial condition is given by
\begin{equation}
  y_0 = \frac{1}{2}\exp\Big(-\big((x-1)^2 + (y-1)^2 + \frac{i}{2}(x-1)\big)\Big) \label{inCond}.
\end{equation}
In our computations, we've used $\Delta t = \frac{1}{2^6} = 0.015625$ and $h = 0.64851$, where 2000 polygons were used to approximate the domain. Figure \ref{fig1} illustrates the state of the numerical solution at different times, while Figure \ref{fig2} left shows the evolution of the energy with time. In this case the decay is exponential, as expected from Theorem \ref{theorem 2.1}.

\begin{figure}[tb]
  \centering
    \begin{subfigure}[t]{0.4325\textwidth}
    \includegraphics[width=\textwidth]{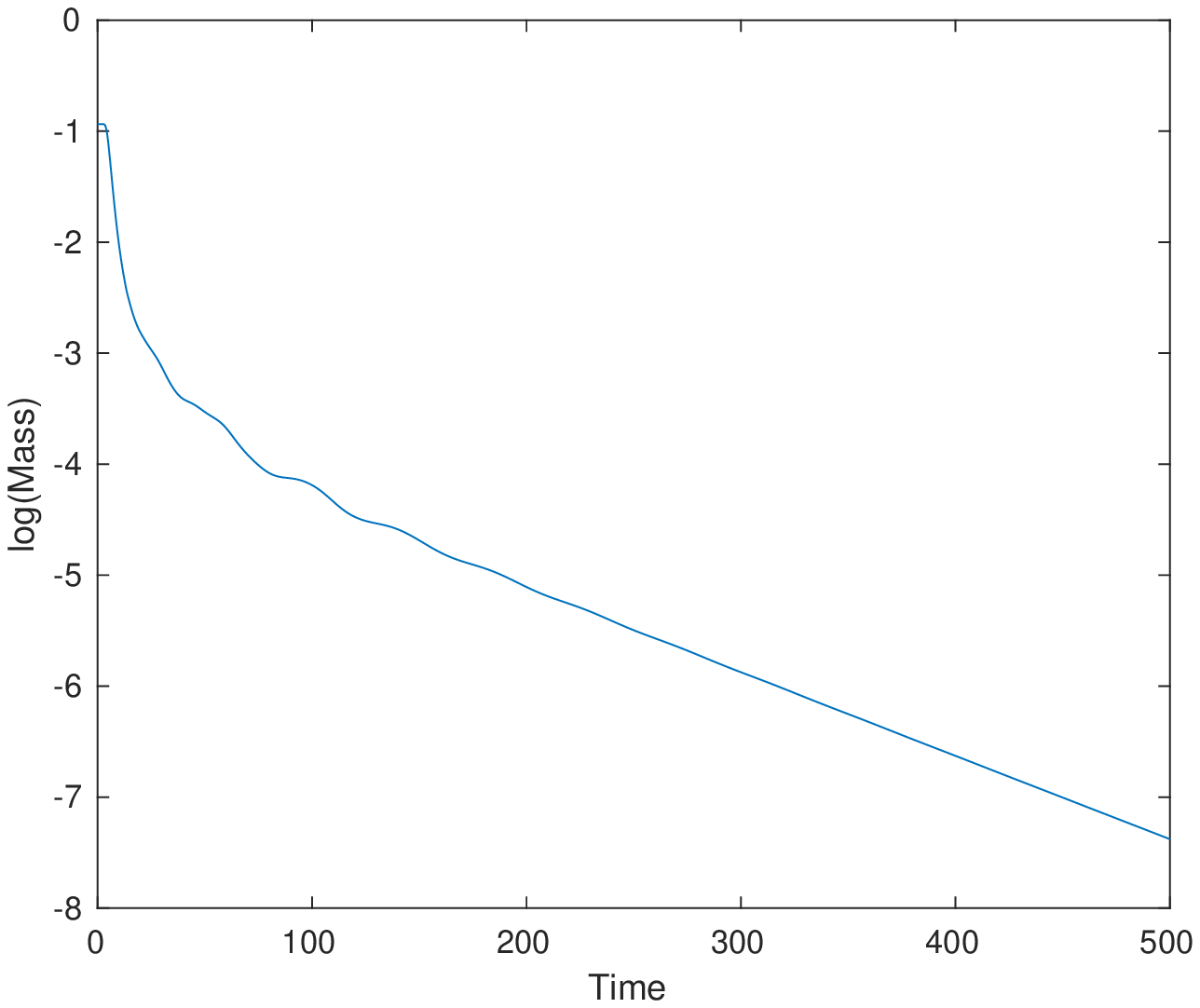}
    \end{subfigure}
      \begin{subfigure}[t]{0.45\textwidth}
    \includegraphics[width=\textwidth]{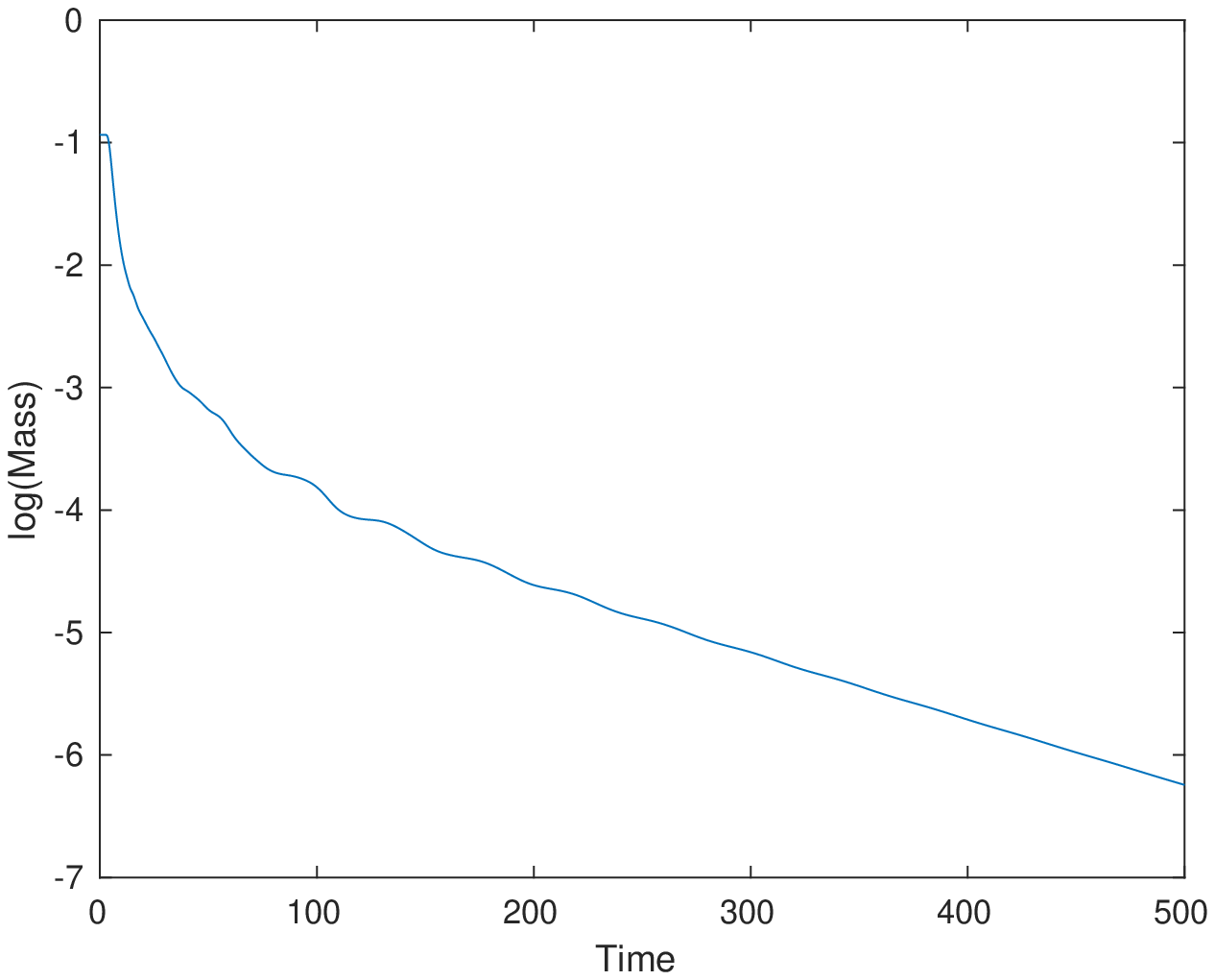}
      \end{subfigure}
      \caption{Energy decays for both examples. Left: decay for Example I. Right: decay for Example II.}
      \label{fig2}
\end{figure}

\subsection{Example II} As a second experiment, we will repeat Example I but using $p=2$, $T=500$, and the damping function
\begin{equation*}
  a(x,y) =
  \begin{cases}
    \big(\exp(\sqrt{x^2+y^2} - 8) - 1 \big)^2, \quad 8^2 \leq x^2 + y^2 \leq 10^2 \\
    0, \qquad \qquad \qquad \qquad \qquad \qquad \text{ in other case.}
  \end{cases}
\end{equation*}
This function also fulfills condition \eqref{anabla}. Figure \ref{fig2} right shows the time evolution of the energy. The decay in this case is also exponential, replicating the theoretical result \eqref{theorem 2.1} proved in the previous sections.

\subsection{Example III} We will now consider an exterior domain, as stated in Section \ref{section3}. The new domain $\Omega$ will be defined as:
\begin{equation*}
  \Omega = \{(x,y) \in \mathbb{R}^2: 5 \leq \sqrt{x^2 + y^2} \leq 20\},
\end{equation*}
while the effective damping subset will be given by
\begin{equation*}
  \omega = \{(x,y) \in \mathbb{R}^2: \sqrt{x^2+y^2} \geq 17\}.
\end{equation*}
The initial condition to be used is
\begin{equation*}
  y(x,0) = \exp\Big(-\big(x^2 + (y-10)^2 + \frac{i}{2}x \big)\Big).
\end{equation*}
\begin{figure}[h]
  \centering
    \begin{subfigure}[t]{0.425\textwidth}
    \includegraphics[scale=0.6,bb=0 0 30 30]{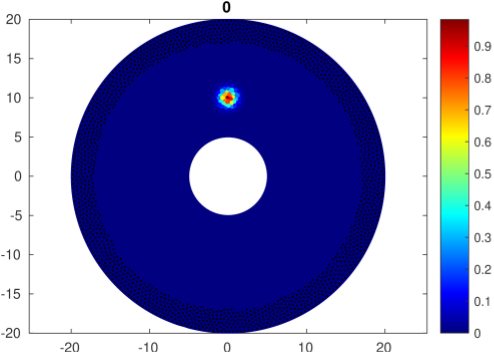}
    \end{subfigure}
      \begin{subfigure}[t]{0.425\textwidth}
    \includegraphics[width=\textwidth]{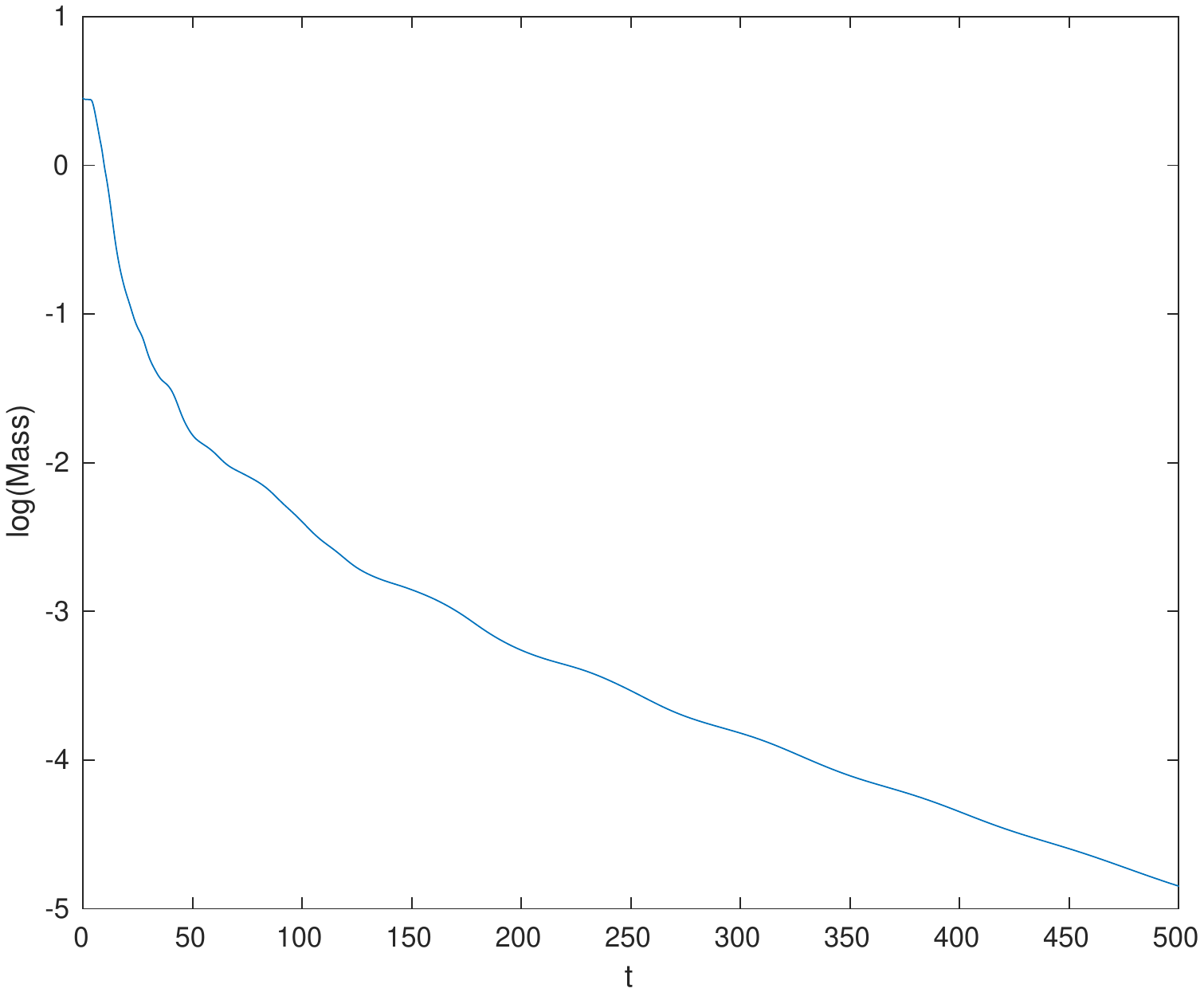}
      \end{subfigure}
      \caption{Results for the experiment with an exterior domain. Left: the initial condition. Right: semi-log plot for the time-evolution of the mass function.}
      \label{fig3}
\end{figure}
For these calculations, we've done $\Delta t = \frac{1}{2^6} = 0.0156$, and the domain was approximated using $5000$ polygons with $h = 0.76172$. Figure \ref{fig3} illustrates the initial condition and the time evolution of the mass functional. Its decay follows an exponential trend, as expected.

\subsection{Example IV} As a final experiment, we will repeat the previous case but using the following domain
\begin{equation*}
  \Omega = \{(x,y) \in \mathbb{R}^2: 7 \leq \sqrt{x^2 + y^2} \leq 20\}.
\end{equation*}
The effective damping subset will be given by $\omega = \{(x,y) \in \mathbb{R}^2: \sqrt{x^2 + y^2} \geq 17 \land \alpha \in (-\pi,0)\}$, where $\alpha$ is the angle of the point $(x,y)$ with respect to the positive $x$ axis. This is equivalent to the geometric condition \eqref{gamma0} for a point $x^0 = (0,y)$ such that  $y \rightarrow +\infty$.
\begin{figure}[h]
  \centering
    \begin{subfigure}[t]{0.5\textwidth}
    \includegraphics[width=\textwidth]{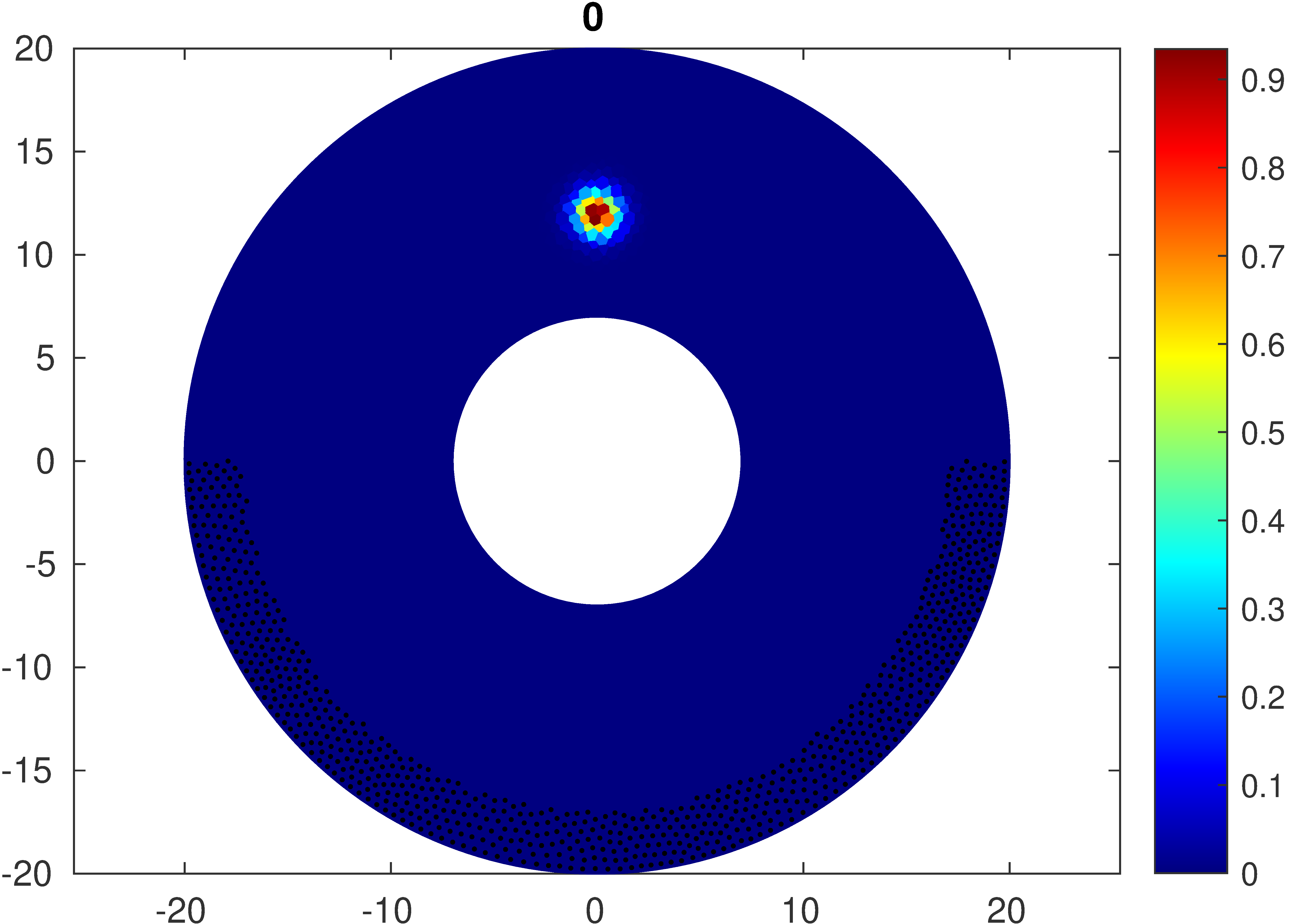}
    \end{subfigure}
      \begin{subfigure}[t]{0.425\textwidth}
    \includegraphics[width=\textwidth]{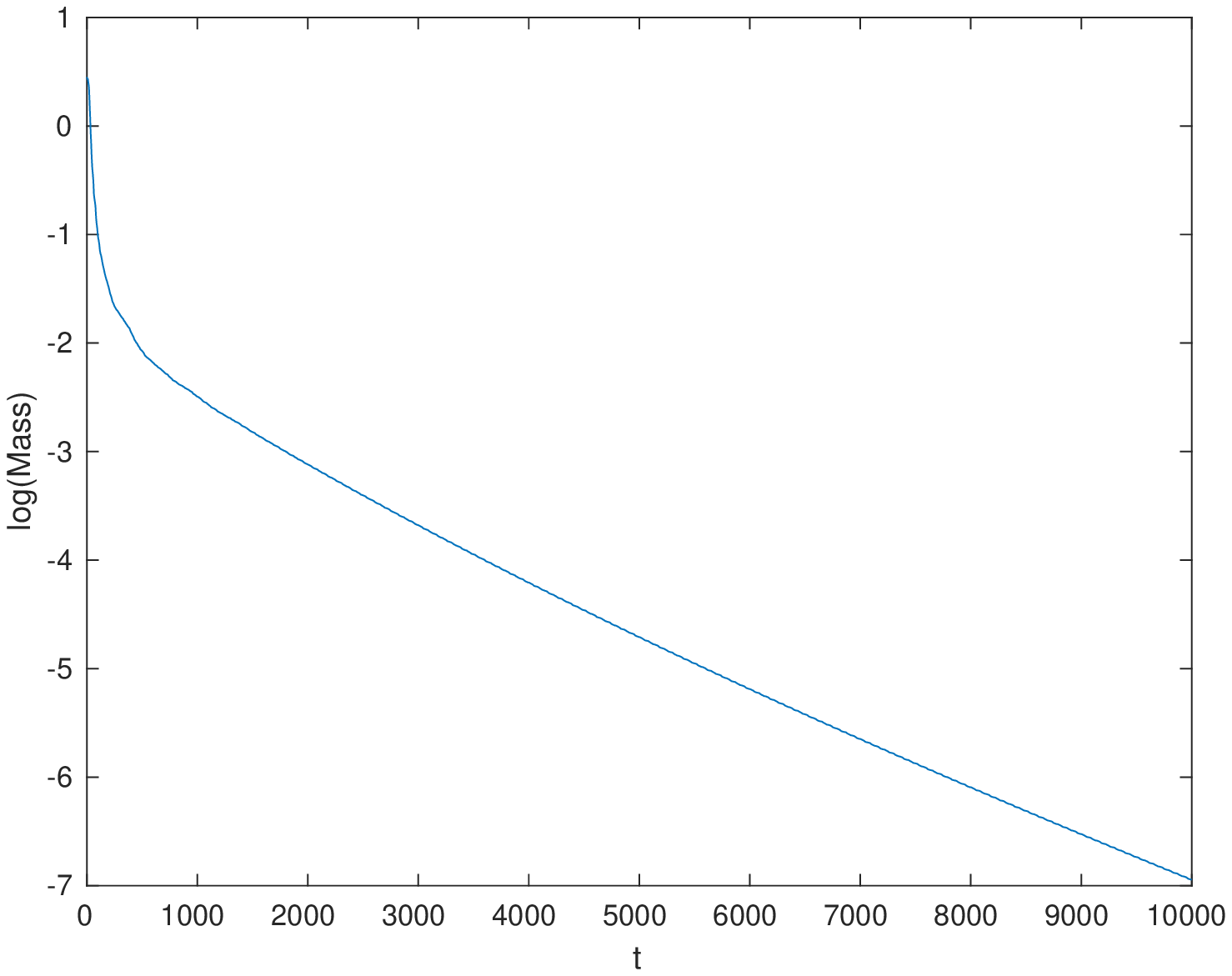}
      \end{subfigure}
      \caption{Left: the initial condition. Black dots denote the cells where the damping function is acting effectively. Right: time evolution of the mass functional, at semi-log scale.}
      \label{fig4}
\end{figure} \\

For our calculations, we've used $\Delta t = \frac{1}{2^5} = 0.0312$, $T = 10000$, and $h = 0.80958$ for a domain approximated using $5000$ polygons. The left panel of Figure \ref{fig4} shows the initial condition and the zone where the damping is acting effectively; while the right panel shows the decay of the Mass funcional in semi-log scale. We can clearly see the exponential decay rate, as expected from Section \ref{section3}.



\section{Final Conclusions}
	 	 \setcounter{equation}{0}
The following table summarizes the new contributions of the present paper compared with the works cited in the introduction.

\newpage
		\begin{center}
\hspace*{.9cm}\quad 		\begin{adjustbox}{max width=.9\textwidth}

		\begin{tabular}{@{\extracolsep{\fill}}|M{1.8cm}|M{1.5cm}|M{3.5cm}|M{2.5cm}|M{5.5cm}|}
			
			\hline 		\multicolumn{5}{|c|}{\small Summary of the literature with respect to problem $ iu_t+\Delta u-f(u)+\text{``damping''}=0 $} \\ \hline
			Authors	& $ f(u) $	&  Damping     & Setting & Tools/Comments
			\\ \hline
			Tsutsumi (1990) \cite{Tsutsumi3}& $ f(|u|^2)u $	&  $ i a\,u $
			
			$ a=\text{ constant} $  & Bounded domain in  $ \mathbb{R}^N $
			& \hspace*{-.5cm}  \begin{minipage} [t] {0.4\textwidth}
				\vspace{-\topsep}\begin{itemize}
					\setlength{\parskip}{0pt}
					\setlength{\itemsep}{0pt plus 1pt}

					\done  Restrictive values of $ p $ acting on
					{\small $ |f(s)|$ and $|D^\alpha f(s) s^{\alpha-1}|.$}
					%
					\done  Exponential decay in $ H_0^1, H^2, $ {\small$ H^{{2k}}, k>N/4 $ level.}
					\crossed smoothing effect
					\crossed unique continuation
				\end{itemize}
			\end{minipage}
			\\ \hline
			
			Lasiecka and Triggiani (2006) \cite{Lasiecka5}& $ 0 $	&  $ \displaystyle \frac{\partial u}{\partial \nu}= ig(u) $ on $ \Gamma_1$   & Bounded domain in  $ \mathbb{R}^N $ with $ \Gamma=\Gamma_0 \cup \Gamma_1 $
			& \hspace*{-.5cm}  \begin{minipage} [t] {0.4\textwidth}
				\vspace{-\topsep}\begin{itemize}
					\setlength{\parskip}{0pt}
					\setlength{\itemsep}{0pt plus 1pt}
					\done  uniform decay rates at the $ L^2 $ level
					\crossed smoothing effect
					\crossed unique continuation
				\end{itemize} \vspace{-\topsep}
			\end{minipage}
			\\ \hline
			Dehman et al. (2006) \cite{Dehman}& $ P^\prime(|u|^2)u $	&  {\small $-a(x)(1-\Delta)^{-1}a(x)u_t $}  & Compact $ 2 $-dimensional Riemannian manifold without boundary
			& \hspace*{-.5cm}  \begin{minipage} [t] {0.4\textwidth}
				\vspace{-\topsep}\begin{itemize}
					\setlength{\parskip}{0pt}
					\setlength{\itemsep}{0pt plus 1pt}
					\done  exponential decay at the  $ H^1 $ level
					\done microlocal analysis
					\done pseudo-differential dissipation
					\done Strichartz estimates \cite{burq2004strichartz}
					\done unique continuation (assumption)
					\crossed smoothing effect
				\end{itemize} \vspace{-\topsep}
			\end{minipage}
			\\ \hline
			Aloui and Kenissi (2008) \cite{aloui2007stabilizationexterior}& $ 0 $	&
			$ i a(x)u $
			
			&  $ N $-exterior domain  & \hspace*{-.5cm}  \begin{minipage} [t] {0.4\textwidth}
				\vspace{-\topsep}\begin{itemize}
					\setlength{\parskip}{0pt}
					\setlength{\itemsep}{0pt plus 1pt}
					\crossed smoothing effect
					\crossed unique continuation
					\done uniform local energy estimates at the $ L^2 $ level
				\end{itemize} \vspace{-\topsep}
			\end{minipage}
			\\ \hline
			Aloui (2008) \cite{aloui2008smoothing}, \cite{aloui2008smoothingbounded}& $ R_0 $	&
			\cite{aloui2008smoothing}: $ i a(x)(1-\Delta)^{\alpha}a(x)u; $	\cite{aloui2008smoothingbounded}: $ i a(x)(-\Delta)^{\alpha}a(x)u, $
			$ 0 < \alpha \le \frac12 $
			& \cite{aloui2008smoothing}: Compact $ N $-dimensional Riemannian manifold without boundary; \cite{aloui2008smoothingbounded}: bounded domain in $ \mathbb{R}^N $  & \hspace*{-.5cm}  \begin{minipage} [t] {0.4\textwidth}
				\vspace{-\topsep}\begin{itemize}
					\setlength{\parskip}{0pt}
					\setlength{\itemsep}{0pt plus 1pt}
					\done $ R_0 $ is a pseudo-differential dissipation		\done smoothing effect
					\crossed unique continuation
					\crossed stabilization
				\end{itemize} \vspace{-\topsep}
			\end{minipage}
			\\ \hline
			Cavalcanti et al. (2009) \cite{Cavalcanti}& $ |u|^2\,u $	&  $ i a(x)u $  &  $ \mathbb{R}^2 $
			& \hspace*{-.5cm}  \begin{minipage} [t] {0.4\textwidth}
				\vspace{-\topsep}\begin{itemize}
					\setlength{\parskip}{0pt}
					\setlength{\itemsep}{0pt plus 1pt}
					
					\done  exponential decay at the  $ L^2 $ level
					\done unique continuation (proved for undamped problem)
					\done smoothing effect in $ H^{1/2} $ norm
					
				\end{itemize} \vspace{-\topsep}
			\end{minipage}
		
						\\ \hline

						Laurent (2010) \cite{Laurent}& {\small$ (1+|u|^2)u $}	&  {\small $-a(x)(1-\Delta)^{-1}a(x)u_t $}  &  Some compact manifolds of dimension 3
						& \hspace*{-.5cm}  \begin{minipage} [t] {0.4\textwidth}
							\vspace{-\topsep}\begin{itemize}
								\setlength{\parskip}{0pt}
								\setlength{\itemsep}{0pt plus 1pt}
								
								\done  exponential decay at the  $ H^1 $ level for periodic solutions
								\done solutions lie in Bourgain spaces
								\done the propagation of compactness and regularity
								
								\done microlocal defect measure \cite{gerard1991microlocal}
								\done pseudo-differential dissipation
								\done unique continuation (assumption)
								\crossed smoothing effect
							\end{itemize} \vspace{-\topsep}
						\end{minipage}
						\\ \hline


		\end{tabular}

	\end{adjustbox}							
\end{center}

\begin{adjustbox}{max width=0.9\textwidth}		

		\begin{tabular}{@{\extracolsep{\fill}}|M{2.cm}|M{1.5cm}|M{3.5cm}|M{2.5cm}|M{5.5cm}|}

			\hline 		\multicolumn{5}{|c|}{\small Summary of the literature with respect to problem $ iu_t+\Delta u-f(u)+\text{``damping''}=0 $} \\ \hline
			Authors	& $ f(u) $	&  Damping     & Environment & Tools/Comments
			\\
			\hline

	Cavalcanti et al. (2010) \cite{Cavalcanti2}& $ \pm|u|^2\,u $	&  \begin{enumerate}[(i)]
		\item $ i a(x)u $
		\item $ i b(x)|u|^2u $
	\end{enumerate}  &  $ \mathbb{R} $
	& \hspace*{-.5cm}  \begin{minipage} [t] {0.4\textwidth}
		\vspace{-\topsep}\begin{itemize}
			\setlength{\parskip}{0pt}
			\setlength{\itemsep}{0pt plus 1pt}
			
			\done  exponential decay at the  $ L^2 $ level (i) and polynomial decay (ii)
			\done unique continuation \cite{zhang1997unique}
			\done smoothing effect in $ H^{1/2} $ norm
			
		\end{itemize} \vspace{-\topsep}
	\end{minipage}
	\\ \hline
						Rosier and Zhang (2010) \cite{Rosier1}& $\lambda u^{\alpha_1}\bar{u}^{\alpha_2},  $
						{\tiny  $ \alpha_1+\alpha_2\ge 2 $}	&
						$ i a^2(x)u $
						&  $ \mathbb{T}^N $
						& \hspace*{-.5cm}  \begin{minipage} [t] {0.4\textwidth}
							\vspace{-\topsep}\begin{itemize}
								\setlength{\parskip}{0pt}
								\setlength{\itemsep}{0pt plus 1pt}
								
								\done   internal stabilization
								\done unique continuation \cite{zhang1997unique}
								\done smoothing effect \cite{bourgain1999global}
								
							\end{itemize} \vspace{-\topsep}
						\end{minipage}
			
						\\ \hline
			Özsar\i\, et al. (2011) \cite{Ozsar}& $ |u|^pu $	&  $ i a\,u $
						
						$a= $ constant
						&  Bounded domain in $ \mathbb{R}^N $ subject to inhomogeneous Dirichlet/ Neumann
						
						& \hspace*{-.5cm}  \begin{minipage} [t] {0.4\textwidth}
							\vspace{-\topsep}\begin{itemize}
								\setlength{\parskip}{0pt}
								\setlength{\itemsep}{0pt plus 1pt}
								
								\done   exponential stabilization at the $ H^1, H^2$ (smallness on the initial data) level
								\done  monotone operator theory
								\crossed unique continuation
								\crossed smoothing effect
								
							\end{itemize} \vspace{-\topsep}
						\end{minipage} \\ \hline
			Özsar\i\, et al. (2012) \cite{Ozsari2}& $ |u|^pu $	&  $ i a\,u $
			
			$a= $ constant
			&  Bounded domain in $ \mathbb{R}^N $ with Dirichlet control
			
			& \hspace*{-.5cm}  \begin{minipage} [t] {0.4\textwidth}
				\vspace{-\topsep}\begin{itemize}
					\setlength{\parskip}{0pt}
					\setlength{\itemsep}{0pt plus 1pt}
					
					\done   exponential stabilization at the $ H^1 $ level
					\done $ 0 < p < \frac{4}{N+2} $
					\done maximal monotone operator theory
					\crossed unique continuation
					\crossed smoothing effect
					
				\end{itemize} \vspace{-\topsep}
			\end{minipage}
			\\ \hline
			Bortot et al. (2013) \cite{Bortot}& 0	&  $ i a(x)g(u) $  & Compact $ N $-dimensional Riemannian manifold with smooth boundary
			& \hspace*{-.5cm}  \begin{minipage} [t] {0.4\textwidth}
				\vspace{-\topsep}\begin{itemize}
					\setlength{\parskip}{0pt}
					\setlength{\itemsep}{0pt plus 1pt}

					\done  observability inequality for the linear problem
					\crossed smoothing effect
					\crossed unique continuation
				\end{itemize} \vspace{-\topsep}
			\end{minipage}
			
			\\ \hline 	
			Aloui (2013) \cite{Aloui2}& $ 0 $	&
			$ i a(x)(-\Delta)^{1/2}a(x)u $
			
			& $ \Omega_0=\mathcal{O}\cap B, $ where   $ B $ is a bounded domain in $ \mathbb{R}^N $
			and $ \mathcal{O} $ is the union of a finite number of bounded strictly convex bodies
			&\hspace*{-.5cm}
			\begin{minipage} [t] {0.4\textwidth}
				\vspace{-\topsep}\begin{itemize}
					\setlength{\parskip}{0pt}
					\setlength{\itemsep}{0pt plus 1pt}
					\done smoothing effect
					\crossed unique continuation
					\done  exponential decay at the  $ L^2 $ level on $ \Omega_0 $
				\end{itemize} \vspace{-\topsep}
			\end{minipage}
			\\ \hline
Özsar\i\, (2013) \cite{Ozsari3}& $ \alpha|u|^pu $	&  $ i\beta|u|^qu, \beta>0 $

& Bounded domain in  $ \mathbb{R}^N $ with $ \Gamma=\Gamma_0 \cup \Gamma_1 $
&\hspace*{-.5cm}
\begin{minipage} [t] {0.4\textwidth}
	\vspace{-\topsep}\begin{itemize}
		\setlength{\parskip}{0pt}
		\setlength{\itemsep}{0pt plus 1pt}
		\crossed smoothing effect
		\crossed unique continuation
		\done  exponential decay at the  $ L^2 $ level when $ q=0 $ and \break $ 0 < p < \frac{4}{N+2} $
	\end{itemize} \vspace{-\topsep}
\end{minipage}
\\ \hline
			
			Bortot and Cavalcanti (2014) \cite{Bortot1}& 0	&  $ i a(x)g(u) $  &  Exterior domain and non compact $ n $-dimensional Riemannian manifold & \hspace*{-.5cm}  \begin{minipage} [t] {0.4\textwidth}
				\vspace{-\topsep}\begin{itemize}
					\setlength{\parskip}{0pt}
					\setlength{\itemsep}{0pt plus 1pt}
					\done  exponential decay at the  $ L^2 $ level
					\done smoothing effect \cite{burq2004nonlinear}
					\done unique continuation \cite{Triggiani_Xu}
				\end{itemize} \vspace{-\topsep}
			\end{minipage}
			\\ \hline

		\end{tabular}
		
		%
	\end{adjustbox}
	
	\begin{center}
		\begin{adjustbox}{max width=0.9\textwidth}		

			\begin{tabular}{@{\extracolsep{\fill}}|M{2.cm}|M{1.5cm}|M{3.5cm}|M{2.5cm}|M{5.5cm}|}

				\hline 		\multicolumn{5}{|c|}{\small Summary of the literature with respect to problem $ iu_t+\Delta u-f(u)+\text{``damping''}=0 $} \\ \hline
				Authors	& $ f(u) $	&  Damping     & Environment & Tools/Comments
				\\ \hline	
				Natali (2015) \cite{Natali}& $ |u|^{\alpha-1}u $	&  $ i b(x)g(u) $  &   $ \mathbb{R}$ & \hspace*{-.5cm}  \begin{minipage} [t] {0.4\textwidth}
					\vspace{-\topsep}\begin{itemize}
						\setlength{\parskip}{0pt}
						\setlength{\itemsep}{0pt plus 1pt}
						\done smoothing effect in $ H^{1/2} $ norm
						\done unique continuation
						\done  exponential decay at the  $ L^2 $ level for $ \alpha=3,5 $
					\end{itemize} \vspace{-\topsep}
				\end{minipage}

				\\
				\hline

				Natali (2016) \cite{Natali1}& $ |u|^{2}u $	&  $ i a(x)g(u) $  &   $ \mathbb{R}^2$ & \hspace*{-.5cm}  \begin{minipage} [t] {0.4\textwidth}
					\vspace{-\topsep}\begin{itemize}
						\setlength{\parskip}{0pt}
						\setlength{\itemsep}{0pt plus 1pt}
						\done smoothing effect in $ H^{1/2} $ norm
						\done unique continuation \cite{Cavalcanti}
						\done  exponential decay at the  $ L^2 $ level
					\end{itemize} \vspace{-\topsep}
				\end{minipage}
				
				\\ \hline	
				Kalantarov and Özsar\i\, (2016) \cite{Kalantarov}& $ |u|^{p}u $	&  $ i a\,u $
				
				$ a=$ constant  &   $ (0,\infty)$ & \hspace*{-.5cm}  \begin{minipage} [t] {0.4\textwidth}
					\vspace{-\topsep}\begin{itemize}
						\setlength{\parskip}{0pt}
						\setlength{\itemsep}{0pt plus 1pt}
						\crossed smoothing effect
						\crossed unique continuation
						\done  decay rates are determined according to the relation between the powers of the nonlinearities.
					\end{itemize} \vspace{-\topsep}
				\end{minipage}
				
				\\ \hline
				
				Cavalcanti et al. (2017) \cite{Cavalcanti0}& $ |u|^{p}u $	&  $ i\lambda(x,t)u $  &   $ \mathbb{R}^N$ & \hspace*{-.5cm}  \begin{minipage} [t] {0.4\textwidth}
					\vspace{-\topsep}\begin{itemize}
						\setlength{\parskip}{0pt}
						\setlength{\itemsep}{0pt plus 1pt}
						\crossed smoothing effect
						\crossed unique continuation
						\done  exponential decay at the  $ L^2 $ level and $ H^1\cap L^{p+2} $ level
					\end{itemize} \vspace{-\topsep}
				\end{minipage}
				\\ \hline
				
{Burq and Zworski (2017)} \cite{burq2017rough}& $ 0 $	&  $ ia(x)u $  &   2-Tori & \hspace*{-.5cm}  \begin{minipage} [t] {0.4\textwidth}
					\vspace{-\topsep}\begin{itemize}
						\setlength{\parskip}{0pt}
						\setlength{\itemsep}{0pt plus 1pt}
						\crossed smoothing effect
						\crossed unique continuation
						\done  exponential decay at the  $ L^2 $ level
					\end{itemize} \vspace{-\topsep}
				\end{minipage}				
				\\ \hline	
				Bortot and Corrêa (2018) \cite{Bortot2017}& $ f(|u|^2)u $	&
				$ i a(x)(-\Delta)^{1/2}a(x)u $&
				
				Bounded domain in $ \mathbb{R}^N $
				
				&\hspace*{-.5cm}
				\begin{minipage} [t] {0.4\textwidth}
					\vspace{-\topsep}\begin{itemize}
						\setlength{\parskip}{0pt}
						\setlength{\itemsep}{0pt plus 1pt}
						\done $ |f(s)|, f^\prime(s)s \leq M $ and $ f(s)s\ge 0 $
						\done smoothing effect \cite{aloui2008smoothingbounded}
						\done unique continuation (proved for undamped problem)
						\done  exponential decay at the  $ L^2 $ level
						
					\end{itemize} \vspace{-\topsep}
				\end{minipage}
				\\ \hline
				
				Cavalcanti et al. (2018) \cite{Cavalcanti4}&
				
				(i) $ |u|^2\,u $
				
				\smallskip
				
				(ii)
				\medskip
				
				$ f(|u|^2)u $
				& (i) $ i\,a(x)u $
				
				\medskip

				(ii)
				
				\medskip
				
				$ i a(x)(-\Delta)^{1/2}a(x)u $ \medskip &
				
				Compact $ N $-dimensional Riemannian manifold without boundary ($ N=2 $ only in the case (i))
				
				&\hspace*{-.5cm}
				\begin{minipage} [t] {0.4\textwidth}
					\vspace{-\topsep}\begin{itemize}
						\setlength{\parskip}{0pt}
						\setlength{\itemsep}{0pt plus 1pt}
						\done $ |f(s)|, f^\prime(s)s \leq M $ and $ f(s)s\ge 0 $
						\done smoothing effect \cite{aloui2008smoothing}
						\done unique continuation
						\done  exponential decay at the  $ L^2 $ level (case (ii))
						\done energy functional goes to zero as time goes to infinity (case (i))
						
					\end{itemize} \vspace{-\topsep}
				\end{minipage}
				\\ \hline
				{\bf Present article}&
				
				$ |u|^p\,u $

				&  $ i\,a(x)u $
				
				&
				
				(i) Bounded domain in $ \mathbb{R}^N $
				\medskip

				(ii) $ \mathbb{R}^N $

\medskip

		(iii) Exterior domain
				
				&\hspace*{-.5cm}
				\begin{minipage} [t] {0.4\textwidth}
					\vspace{-\topsep}\begin{itemize}
						\setlength{\parskip}{0pt}
						\setlength{\itemsep}{0pt plus 1pt}
						\done No restriction for $ p $ and $ N $
						\crossed smoothing effect
						\crossed Strichartz estimates
						\crossed Microlocal analysis
					\done monotone theory
						\done unique continuation \cite{las03}
						\done  exponential decay at the  $ L^2 $
						
					\end{itemize} \vspace{-\topsep}
				\end{minipage}
				\\ \hline
			\end{tabular}
			
			%
		\end{adjustbox}
	\end{center}
	
	\medskip
	
		\subsection{ Strong stability versus uniform energy decay rates} Making use of the assumption \eqref{anabla}, we obtain exponential decay at the $ L^2 $ level.  When (\ref{anabla}) is no longer valid, the constant on the right hand side of (\ref{ident_9}) will depend on $ C=C\left(T,\|y_0\|_{\mathcal{X}}\right)$. In this case, instead of exponential decay rate estimates one just has that the energy $E_0(t)$ goes to zero when $t$ goes to infinity (as in Cavalcanti et al. \cite{{Cavalcanti4}}). Indeed, fix $T_0^\ast > T_0$, where $T_0>0$ considered large enough comes from the unique continuation property. Then, from \eqref{final_inequality}  there exists a constant $C=C(L,T_0^\ast)$ such that
\begin{eqnarray}\label{eqt1}
E_0(0) \leq C(L,T_0^\ast) \int_0^{T_0^\ast}\int_\Omega a(x)\,|y|^2\,dxdt.
\end{eqnarray}

The identity of the energy yields
\begin{eqnarray}\label{eqt2}
\int_0^{T_0^\ast}\int_\Omega a(x)\,|y|^2\,dxdt=-E_0(T_0^\ast) + E_0(0).
\end{eqnarray}

Combining (\ref{eqt1}) and (\ref{eqt2}) and since $E_0(T_0^\ast) \leq E_0(0)$, we infer
\begin{eqnarray*}
	E_0(T_0^\ast) \left(1+C(L,T_0^\ast) \right)\leq C(L,T_0^\ast) E_0(0),
\end{eqnarray*}
from which we conclude that
\begin{eqnarray*}
	E_0(T_0^\ast) \leq \left(\frac{C(L,T_0^\ast)}{1+ C(L,T_0^\ast)}\right)E_0(0),
\end{eqnarray*}
and, consequently, since the map $t \mapsto E_0(t)$ is non-increasing, we deduce
\begin{eqnarray}\label{eqt3}
E_0(T) \leq \gamma_1 E_0(0),~\forall T > T_0,~ \hbox{ where }\gamma_1:= \left( \frac{1}{\tilde C_0 +1}\right),
\end{eqnarray}
and $\tilde {C}_0 = \tilde {C}_0(L, T_0^\ast)$. From the boundedness $||y(t)||_{H_0^1(\Omega)}\leq C\left(T,\|y_0\|_{\mathcal{X}}\right)$ one has $||y(T)||_{H^1(M)} \leq C_1(T)$, and as we have proceed above we conclude that
\begin{eqnarray}\label{eqt4}
E_0(2T) \leq \gamma_2 E_0(T),~\forall T > T_0,~\hbox{ where } \gamma_2:=\left( \frac{1}{\tilde C_1 +1}\right)
\end{eqnarray}
and $\tilde C_1 = \tilde C_1 (C_1(T),T_0^\ast)$. Thus, from (\ref{eqt3}) and (\ref{eqt4}) we arrive at
\begin{eqnarray*}
	E_0(2T) \leq (\gamma_1 \gamma_2)\, E_0(0),~~\forall T > T_0, \hbox{ with }\gamma_1,\gamma_2 <1,
\end{eqnarray*}
and recursively we obtain for all $n\in \mathbb{N}$, that
\begin{eqnarray}\label{eqt5}
E_0(nT) \leq (\gamma_1 \gamma_2 \cdots \gamma_n)\, E_0(0), ~\forall T > T_0, ~ \hbox{ with }\gamma_1,\gamma_2, \cdots, \gamma_n <1.
\end{eqnarray}

Thus, if we assume, by contradiction, that the map $t \mapsto E_0(t)$ is bounded from below, namely, if there exists $\alpha>0$ such that $E_0(t) \geq \alpha$ for all $t>0$, then from (\ref{eqt5}) it follows that $E_0(nT) \leq \gamma^n\, E_0(0)$ for some $\gamma<1$, and we obtain a contradiction.  Consequently $E_0(t)$ goes to zero when t goes to infinity. \quad $\square$

\medskip

From the above, we are adjusted with Liu and Rao final results \cite{Liu}, namely, uniform stability or just uniform and exponential decay rate estimates. However, they exploit the assumption (\ref{anabla}), looking for resolvent estimates, while in our case we are looking for global solutions in $H^1-$level bounded by $||u(t)||_{H_0^1(\Omega)}\leq C\left(\|y_0\|_{\mathcal{X}}\right)$.

\bibliographystyle{abbrv}
\bibliography{myreferences}
\end{document}